\documentclass[lettersize,journal]{IEEEtran}
\usepackage{amsmath,amsfonts}
\allowdisplaybreaks
\usepackage{algorithmic}
\usepackage{algorithm}
\usepackage{array}
\usepackage[caption=false,font=normalsize,labelfont=sf,textfont=sf]{subfig}
\usepackage{textcomp}
\usepackage{stfloats}
\usepackage{url}
\usepackage{amssymb}
\usepackage{verbatim}
\usepackage{graphicx}
\usepackage{xcolor}
\usepackage{booktabs}
\usepackage[percent]{overpic}
\usepackage{multirow}
\usepackage{cite}
\newtheorem{theorem}{\textbf{Theorem}}
\newtheorem{proposition}{\textbf{Proposition}}

\newtheorem{remark}{Remark}

\newtheorem{assumption}{Assumption}
\newcommand{\card}{\mathrm{card}}

\begin{document}

\title{ Decentralized Stability-Constrained Optimal Power Flow\\for Inverter-Based Power Systems}

\author{Shigeng Wang,~\IEEEmembership{Student Member,~IEEE,} and Sijia Geng,~\IEEEmembership{Member,~IEEE}
\thanks{Shigeng Wang and Sijia Geng are with the Department of Electrical and Computer Engineering,
Johns Hopkins University, Baltimore, MD 21218, USA. Corresponding Author: Sijia Geng (e-mail: sgeng@jhu.edu).}
}



\maketitle

\begin{abstract}
Future inverter-dominated power systems feature higher variability and more stressed operating conditions, which motivates the consideration of stability in operational settings. Existing approaches to stability-constrained OPF often rely on eigenvalue calculation, global model information, or dynamic evaluation inside optimization formulation, which are computationally intensive and difficult to scale.
This paper proposes the first decentralized stability-constrained OPF framework for inverter-based power systems. The key novelty lies in the incorporation of a class of algebraic decentralized small-signal stability criteria that admits tractable representations in steady-state variables and is therefore suitable for optimization. The decentralized stability condition is based on local voltage differences and enables clear theoretical and practical economic interpretation of the stability contribution from each inverter. We define a Nodal Stability Shadow Price (NSSP) for each inverter, and characterize the role of these stability constraints through their associated shadow prices, enabling a nodal interpretation of their economic impacts. It is proved that under active-power-only objectives in lossless networks, binding stability constraints may  occur but will admit zero shadow prices if all other operational constraints are inactive. Most importantly, we reveal the importance of considering the opportunity cost of reactive power for inverter-based resources (IBRs) that have limited capacity. When reactive power costs are considered, stability constraints can carry strictly positive shadow prices and admit meaningful economic impacts. Numerical studies on a two-bus system and the IEEE 39-bus system validate the proposed framework and illustrate how stability constraints influence optimal operational decisions and the impacts of control and network parameters such as the reactive power droop for grid-forming inverters.
\end{abstract}


\section{Introduction}

\IEEEPARstart{W}{ith} the increasing integration of inverter-based resources (IBRs), modern power systems are undergoing a fundamental transition from synchronous-machine-dominated systems to power-electronics-dominated systems \cite{dfpscc, geng2022unified}. The dynamic behavior of IBRs is primarily governed by control schemes, and, compared to synchronous machines, they present more stringent operational constraints such as current limitation \cite{rathnayake2026grid}. Future low-inertia IBR-based systems feature higher variability of operating conditions and higher complexity of system dynamics \cite{wu2024approximating, geng2025unified}. Furthermore, with the increasing loading, stability is becoming recognized as an operational concern as it becomes tightly coupled with power dispatch and network constraints.

Optimal power flow (OPF) has long served as a cornerstone of power system operation by enabling economically efficient dispatch while respecting network and device constraints \cite{opf79}. Modern power system operation has already considered various dynamic aspects through ancillary services, such as frequency regulation \cite{frereg} and fast frequency response \cite{frefast}, which explicitly account for system dynamics at dispatch levels \cite{operaifac}. Early work on stability-constrained OPF primarily focused on transient stability. For example, \cite{scopf2000} incorporates dynamic stability constraints by embedding time-domain system dynamics into the OPF formulation through numerical discretization of differential equations. This approach provides a general framework for ensuring transient stability within OPF, but typically requires time-domain simulations and introduces a large number of additional constraints, resulting in significant computational complexity. In contrast, small-signal stability—despite being critical to stable operation—has largely remained an offline assessment, typically verified only after optimal operating points are determined. Although recent studies have attempted to integrate stability considerations into OPF through eigenvalue-based constraints \cite{eig2006, eig2024}, such approaches are often computationally intensive, require a global model, and are difficult to scale. These limitations become particularly pronounced in inverter-based power systems, 
which motivates stability-aware operational frameworks that are both scalable and interpretable.

In this paper, we develop a stability-constrained OPF framework that enables a unified treatment of optimality and small-signal stability for inverter-based power systems.
A key challenge in integrating small-signal stability into OPF lies in identifying sufficient stability conditions that are suitable for optimization formulation. 
Decentralized small-signal stability criteria that algebraically depend on quantities available at individual buses offer a natural pathway in this regard \cite{pesgm,tcns}.
\cite{pesgm} developed such decentralized stability criteria in the form of algebraic inequality constraints involving local voltage magnitudes and network parameters for certifying small-signal stability of an inverter-based power system. These conditions can be evaluated using (steady-state) quantities available at individual buses and their neighbors. As a result, they can be directly embedded into OPF formulations as constraints and avoid explicit reliance on a global model, eigenvalue computation, and dynamic evaluation.




Furthermore, beyond computational advantages, the decentralized stability conditions in \cite{pesgm} can be associated with individual devices and their local interactions. 
In particular, when embedded into OPF, such constraints naturally admit dual interpretations, allowing quantification of the marginal economic impacts of stability requirements through shadow prices. This facilitates a device-level economic assessment of stability provision, revealing how operating decisions of each device contribute to overall system stability and how stability requirements trade off against economic objectives. 

It is worth mentioning that 
a large body of work has developed small-signal stability conditions with decentralization flavors, which, however, present various difficulties in explicit incorporation and interpretation in optimization. Transfer-function-based analysis was used in \cite{huangtsg, huangtps, vinniauto},
which support heterogeneous component models. Nevertheless, they rely on frequency-domain descriptions and require knowledge of the full network model, which makes the resulting conditions difficult to be included as constraints of steady-state operating  variables in OPF. 
Another line of work establishes stability through spectral characterizations of the closed-loop system, where stability is inferred from eigenvalue conditions of the system dynamics \cite{dominic, andysun}. In addition, energy-based and passivity-based approaches provide rigorous stability guarantees using Lyapunov or dissipativity arguments \cite{liutps}, and can accommodate nonlinear and heterogeneous models. However, these methods are formulated in terms of dynamical variables and storage functions, and therefore do not admit explicit algebraic representations in steady-state variables.

In summary, several major contributions are made in this paper as summarized below.
\begin{enumerate}
\item  To the best of our knowledge, this paper presents the first decentralized stability-constrained OPF framework, enabling systematic and rigorous consideration of small-signal stability in operational decision-making for inverter-based power systems.
\item Theoretical characterizations of decentralized stability constraints in OPF are developed through their associated dual variables. A Nodal Stability Shadow Price (NSSP) is proposed for interpreting economic impacts of stability requirements, enabling explicit attribution of stability contributions from each device to the system.
\item We reveal that binding stability constraints can exhibit zero shadow prices under specific conditions, and positive shadow prices can arise when realistic cost considerations for reactive power generation are incorporated into the formulation.
\item The proposed decentralized stability-constrained OPF is demonstrated through benchmark systems, highlighting its practical importance for inverter-based power systems where stability constraints impact optimal operating decisions and lead to stability–economics trade-offs. 
\item Finally, key observations are made on considering the opportunity cost of reactive power for IBRs that have limited capacity, and the important role of the reactive power droop parameter on stability for grid-forming inverters. 
\end{enumerate}

\section{Preliminaries}

\subsection{Network Model and Device Model}

The connected inverter-based electrical network is represented by a graph $\mathcal H=(\mathcal N,\mathcal E)$,
where $\mathcal N$ denotes the set of buses and $\mathcal E$ denotes the set of
transmission lines. Let $\mathcal G\subseteq\mathcal N$ and
$\mathcal L=\mathcal N\setminus\mathcal G$ denote the sets of inverter buses and
load buses, respectively. For each bus $i\in\mathcal N$, let $\mathcal{N}_i$ denote the set of its neighboring buses in the network.

\begin{assumption}
    The electrical network is lossless, i.e., the network admittance
matrix satisfies $Y=G+\mathrm{j}B$ with $G=0$.\label{aslossless}
\end{assumption}

\begin{assumption}
\label{pi/2}
    The steady-state operating points for voltage phase angles satisfy
    \[
    |\theta_i - \theta_j| < \frac{\pi}{2}, \qquad \forall (i,j)\in\mathcal E.
    \]
\end{assumption}

Dynamic states are associated with inverter buses. Next we introduce the dynamic model of the grid-forming inverters.  
All quantities are expressed in per unit unless otherwise specified. 
For each inverter $i\in\mathcal G$, $P_i^0,Q_i^0,V_i^0,\theta_i^0$ denote its nominal setpoints for active power, reactive power, voltage magnitude, and angle, respectively, and $\omega_i^0=1$ represents the nominal frequency.
$\omega_i$ and $V_i$ denote the frequency
and voltage magnitude at inverter bus $i$, respectively. Define the deviations $\Delta\omega_i=\omega_i-\omega_i^0$ and
$\Delta V_i=V_i-V_i^0$. The inverter dynamics are described by
\begin{equation}
    \begin{cases}
    \dot{\theta}_i = \omega_b \Delta\omega_i,\\[2pt]
    \dot{\Delta\omega}_i =
    -\dfrac{1}{\tau^p_i}\Delta\omega_i
    + \dfrac{m^p_i\beta^p_i}{\tau^p_i}\bigl(P_i^0-P_i\bigr),\\[6pt]
    \dot{\Delta V_i} =
    -\dfrac{1}{\tau^q_i}\Delta V_i
    + \dfrac{m^q_i\beta^q_i}{\tau^q_i}\bigl(Q_i^0-Q_i\bigr),
    \end{cases}
    \quad \forall i\in\mathcal G,
    \label{eq:dynamics}
\end{equation}
where $\omega_b$ is the base angular frequency (e.g., $\omega_b = 2\pi \cdot 60$ rad/s).
$P_i$ and $Q_i$
represent the active and reactive power generation. The parameters $m_i^p, 
m_i^q>0$ are the droop coefficients associated with the $P$--$\omega$ and
$Q$--$V$ control channels, respectively. $\tau_i^p, \tau_i^q>0$ are the time constants of the low-pass filters for measured active and reactive powers, respectively, while $\beta_i^p, \beta_i^q>0$
denote the corresponding DC gains of the filters. 

\subsection{Decentralized Small-Signal Stability Criteria}

Small-signal stability is typically characterized by the eigenvalues of the system matrix. 
However, such criteria rely on global system representations and do not yield explicit algebraic constraints. 
To facilitate optimization-based formulations, we focus on (sufficient) stability conditions that admit decentralized algebraic representations in steady-state variables.

In particular, we consider a class of algebraic constraints that depend only on steady-state voltage differences, expressed as,
\begin{equation}
h_\ell(V):= \alpha_\ell^\top V - \Gamma_\ell \le 0, 
\qquad \ell = 1,\dots,L,
\label{eq:stab-nbus-thm}
\end{equation}
where $L$ is the number of constraints. Each vector $\alpha_\ell \in \mathbb{R}^{|\mathcal G|}$ satisfies 
$\alpha_\ell^\top \mathbf{1} = 0$. $\Gamma_\ell > 0$ is a constant determined by system parameters.

As a particular instance of \eqref{eq:stab-nbus-thm}, \cite{pesgm} provides a decentralized voltage-difference condition for the inverter-based power systems through an improved passivity-based argument. The system is small-signal stable
if for each inverter bus $i\in\mathcal G$, the following condition holds,
\begin{equation}
\max_{j\in\mathcal{N}_i^{\mathrm{red}}} (V_j - V_i) 
\le \frac{1}{2 m_i^q \beta_i^q \left|B^{\mathrm{red}}_{ii}\right|},
\label{eq:pesgm_condition}
\end{equation}
where $B^{\mathrm{red}}$ is the susceptance matrix of the Kron reduced network and $\mathcal{N}_i^{\mathrm{red}}$ denotes the set of neighboring buses in the reduced network\footnote{Note that this condition relies on an assumption of negligible transfer conductance in the reduced network, which is commonly adopted for high-voltage transmission systems. For more details, please refer to \cite{pesgm}.}. To derive the condition, the system is modeled as a feedback interconnection, followed by a loop transformation and the use of Gershgorin’s Circle Theorem.

The max constraint in \eqref{eq:pesgm_condition} admits an equivalent representation as a set of split inequalities,
\begin{equation}
   V_j - V_i \le \Gamma_i, \ \forall j \in \mathcal N_i^{\mathrm{red}}. \label{eq:pairwise}
\end{equation}




\section{Decentralized Stability–Constrained OPF}
In this section, we incorporate the voltage-difference stability constraints into an AC-OPF framework in which steady-state operating decisions are optimized subject to both operational constraints and small-signal stability criteria. Stability conditions of the form \eqref{eq:stab-nbus-thm} provide a tractable algebraic representation and explicit dependence on steady-state variables, which enables direct integration into the OPF problem, thereby linking small-signal stability with optimization of steady-state operating points.


We consider an AC-OPF problem with decision variables 
$x:=\big(\{P_{G,i},Q_{G,i}\}_{i\in\mathcal G},\,\{V_i,\theta_i\}_{i\in\mathcal N}\big)$, 
where $P_{G,i}$ and $Q_{G,i}$ denote the active and reactive power generation at bus $i$, and $P_{D,i}$ and $Q_{D,i}$ denote the corresponding load demands. 
For buses $i\in\mathcal N\setminus\mathcal G$ without generation, we set $P_{G,i}=Q_{G,i}=0$. We fix the voltage angle at the reference bus as $\theta_1 = 0$ to eliminate angle invariance.

We consider a general quadratic objective function $J_\mathrm{obj}$ representing generation cost, 
\begin{equation}
J_\mathrm{obj}
=
\sum_{i \in \mathcal G}
\left(
c_i P_{G,i}^2 + b_i P_{G,i} + a_i + d_i Q_{G,i}^2
\right),
\label{eq:general_cost}
\end{equation}
where $a_i, b_i, c_i,$ and $d_i \ge 0$ are cost coefficients. 
In many practical applications, $d_i = 0$ and the objective depends only on active power.

The decentralized stability-constrained OPF problem is formulated in $\mathcal P_1$.
The AC power flow equations, generator operating limits, voltage 
magnitude limits, and branch flow constraints are imposed on the network. 
The stability conditions are defined on the inverter voltage vector $V=(V_i)_{i\in\mathcal G}$ and 
parameterized by the reduced network. 

\begin{align}
(\mathcal P_1)\quad &
\min_{\{P_{G,i},Q_{G,i}\}_{i\in\mathcal G},\,\{V_i,\theta_i\}_{i\in\mathcal N}}
\quad J_\mathrm{obj}
\nonumber\\
\mathrm{s.t.}\quad
&
\theta_1 = 0,\nonumber\\
&
P_{G,i}-P_{D,i}
=
-\sum_{j\in\mathcal N_i}
\,V_iV_jB_{ij}\sin(\theta_i-\theta_j),
\ \  \forall i\in\mathcal N,
\nonumber\\
&
Q_{G,i}-Q_{D,i}
=
V_i^2B_{ii}
+
\sum_{j\in\mathcal N_i}
V_iV_jB_{ij}\cos(\theta_i-\theta_j),\nonumber\\
&
\qquad\qquad\qquad\qquad\qquad\qquad\qquad\qquad\qquad \forall i\in\mathcal N,
\nonumber\\
&
\underline P_{G,i}\le P_{G,i}\le \overline P_{G,i},
\quad \forall i\in\mathcal G,\label{Pg}\\
&
\underline Q_{G,i}\le Q_{G,i}\le \overline Q_{G,i},
\quad \forall i\in\mathcal G,\label{Qg}\\
&
\underline V_i\le V_i\le \overline V_i,
\quad \forall i\in\mathcal N,\label{Vg}\\
&
\big(V_i^2B_{ij}-V_iV_jB_{ij}\cos(\theta_i-\theta_j)\big)^2\nonumber\\
&\ + V^2_iV^2_jB^2_{ij}\sin^2(\theta_i-\theta_j)\le \overline S_{ij}^2,
\quad \forall (i,j)\in\mathcal E,\label{Sij}\\
& 
\alpha_\ell^\top V \le \Gamma_\ell,
\quad \ell=1,\dots,L.\nonumber \ \ (\mathrm{Stability \ Constraints} )
\end{align}


The stability constraints enter the OPF formulation as additional inequality constraints and, thereby, can influence the optimal solution. 
These constraints act on voltage variables and do not explicitly appear in the objective function.
In particular, while the objective is driven by power generation, the stability constraints restrict feasible voltage profiles and thus affect the optimal operating point indirectly through the network equations.
From optimization analysis, each stability constraint is associated with a dual variable (i.e., shadow price), which characterizes the sensitivity of the optimal value to perturbations of the corresponding stability limit $\Gamma_\ell$.

At a KKT point, the multiplier associated with the max constraint \eqref{eq:pesgm_condition} can be expressed as the sum of the multipliers of the active split constraints \eqref{eq:pairwise}, due to the subdifferential structure of the max operator (see Proposition~\ref{prop:max_split_equivalence} in the Appendix for the proof). 
To capture the aggregate effect of these constraints at each bus, we define the Nodal Stability Shadow Price (NSSP) at each inverter bus $i\in\mathcal G$ as,
\begin{equation}
\lambda_i^{\mathrm{stab}} := \sum_{j \in \mathcal N^{\mathrm{red}}_i} \lambda_{ij}^{\mathrm{stab}},\label{eq:NSSP}
\end{equation}
where $\lambda_{ij}^{\mathrm{stab}}$ is the Lagrange multiplier associated with the stability constraint linking buses $i$ and $j$. 
By complementary slackness, only binding constraints contribute to this sum.


\section{Economic Interpretation of Voltage-Difference Stability Constraints}\label{sec:economic}

In this section, we analyze the economic implications of the decentralized voltage-difference stability constraints in the OPF problem. Our goal is to characterize when binding stability constraints carry an economic impact on the cost via shadow prices, and to understand how this depends on the interplay between the objective function and the equality constraints of the OPF problem.

\begin{theorem}[Zero shadow prices when only stability constraints are binding in a lossless system with $P$-only costs]
\label{thm:zero_shadow}
Consider the stability-constrained OPF problem $\mathcal P_1$ for a
lossless network. Assume that the objective function depends only on active power
generation, i.e., $d_i=0$ in \eqref{eq:general_cost}. Let $x^\star$ be a
KKT point of $\mathcal P_1$, and let $\mathcal A := \{\ell : h_\ell(V^\star)=0\}$ denote the active set of stability constraints at $x^\star$.

If all operational inequality constraints (i.e., generator operating limits \eqref{Pg}, \eqref{Qg}, voltage 
magnitude limits \eqref{Vg}, and branch flow constraints \eqref{Sij}) are strictly inactive at $x^\star$, and if the
vectors $\{\alpha_\ell\}_{\ell\in\mathcal A}$ are positively linearly
independent,
then the
Lagrange multipliers associated with all binding stability constraints
satisfy
\[
\mu_\ell^\star = 0, \qquad \forall \ell\in\mathcal A.
\]
\end{theorem}

\noindent\textit{Proof:}
For notational convenience, define power injections at bus $i$ to the network as,
\[
P_i^{\mathrm{net}}(V,\theta)
:= -\sum_{j\in\mathcal N_i} V_iV_jB_{ij}\sin(\theta_i-\theta_j),
\]
\[
Q_i^{\mathrm{net}}(V,\theta)
:= V_i^2B_{ii}+\sum_{j\in\mathcal N_i}V_iV_jB_{ij}\cos(\theta_i-\theta_j).
\]

Let $\lambda_P,\lambda_Q\in\mathbb{R}^{|\mathcal N|}$ be the multipliers
associated with the active and reactive power balance constraints,
respectively, and let $\mu_\ell\ge 0$ be the multiplier associated with
the stability constraint $h_\ell(V)\le 0$, $\ell=1,\dots,L$. Since all
operational inequality constraints \eqref{Pg}, \eqref{Qg}, \eqref{Vg} and \eqref{Sij} are inactive at $x^\star$, their
multipliers are zero.

The Lagrangian is
\begin{align}
\mathcal L
=
&J_\mathrm{obj}
+\lambda_P^\top(P_G-P_D-P^{\mathrm{net}})\nonumber\\
&+\lambda_Q^\top(Q_G-Q_D-Q^{\mathrm{net}})
+\sum_{\ell=1}^L \mu_\ell(\alpha_\ell^\top V-\Gamma_\ell).\nonumber
\end{align}

Since $J_\mathrm{obj}$ depends only on $P_G$, stationarity with respect
to $Q_G$ gives
\[
\lambda_Q=0.
\]

Next, for each non-reference bus $i\in\mathcal N\setminus\{1\}$,
stationarity with respect to $\theta_i$ implies
\[
\frac{\partial \mathcal L}{\partial \theta_i}
=
-\sum_{j\in\mathcal N_i}(\lambda_{P,i}-\lambda_{P,j})V_iV_jB_{ij}
\cos(\theta_i-\theta_j)=0.
\]

Since the network is connected and Assumption~\ref{pi/2}
ensures that $|\theta_i-\theta_j|<\pi/2$ for all $(i,j)\in\mathcal E$,
we have $\cos(\theta_i-\theta_j)>0$ on every line. The above equations
therefore define a weighted graph Laplacian system for $\lambda_P$ whose
null space is $\mathrm{span}\{\mathbf 1\}$. Hence, all components of
$\lambda_P$ are identical, and there exists a scalar $\lambda$ such that
\[
\lambda_P=\lambda\mathbf 1.
\]

Stationarity with respect to $V$ yields
\[
-\nabla_V(\lambda_P^\top P^{\mathrm{net}})
-\nabla_V(\lambda_Q^\top Q^{\mathrm{net}})
+\sum_{\ell=1}^L \mu_\ell \alpha_\ell
=0.
\]

Because $\lambda_Q=0$, this reduces to
\[
-\nabla_V(\lambda_P^\top P^{\mathrm{net}})
+\sum_{\ell=1}^L \mu_\ell \alpha_\ell
=0.
\]

Using $\lambda_P=\lambda\mathbf 1$, we have
\[
\lambda_P^\top P^{\mathrm{net}}
=
\lambda\,\mathbf 1^\top P^{\mathrm{net}}
=
\lambda\sum_{i\in\mathcal N}P_i^{\mathrm{net}}.
\]

For a lossless network,
\[
\sum_{i\in\mathcal N}P_i^{\mathrm{net}}=0,
\]
since each line contribution cancels pairwise, that is,
$
-\,V_iV_jB_{ij}\sin(\theta_i-\theta_j)$ and
$-\,V_jV_iB_{ji}\sin(\theta_j-\theta_i)$
sum to zero.

Hence, $\lambda_P^\top P^{\mathrm{net}}\equiv 0$, and
\[
\nabla_V(\lambda_P^\top P^{\mathrm{net}})=0.
\]

Therefore,
\[
\sum_{\ell=1}^L \mu_\ell \alpha_\ell=0.
\]

By complementary slackness, $\mu_\ell=0$ for all
$\ell\notin\mathcal A$. Thus,
\[
\sum_{\ell\in\mathcal A}\mu_\ell\alpha_\ell=0,
\qquad
\mu_\ell\ge 0,\ \forall \ell\in\mathcal A.
\]
Since the vectors $\{\alpha_\ell\}_{\ell\in\mathcal A}$ are positively
linearly independent, it follows that
\[
\mu_\ell=0,\qquad \forall \ell\in\mathcal{A}. \hspace{5.5cm} \blacksquare
\]

\begin{remark}
Theorem~\ref{thm:zero_shadow} establishes a fundamental property of
voltage-difference stability constraints in lossless OPF problems with
$P$-only objectives. Specifically, even when one or more stability
constraints are active at an optimal solution, their associated shadow
prices are necessarily zero if all other operational inequality
constraints are inactive. That is, enforcing stability does not change the optimal value, and the marginal cost of stability is zero. 
\end{remark}

Under the conditions in Theorem~\ref{thm:zero_shadow}, voltage-difference
stability constraints do not introduce any economic trade-off with the
OPF objective. 
However, this reflects a limitation of the formulation rather than a realistic operational feature. In practice, the limited current capability of IBRs couples active and reactive power, so that adjustments in reactive power inevitably impact active power dispatch and should incur an opportunity cost. 

A natural question arises: under what conditions can an (independently) binding stability constraint induce a positive shadow price? We next characterize the structural requirements for this to occur.

\begin{proposition}[Necessary condition for positive shadow prices of stability constraints in OPF]
\label{thm:necessary}
Under the setting of $\mathcal P_1$, let $J_\mathrm{obj}(x;\rho)$
denote an objective parameterized by $\rho$. We consider the case where a single stability constraint is independently binding. Let $s \in \{1,\dots,L\}$ denote its index. Suppose that $x^\star$
is a KKT point 
satisfying
\[
g(x^\star)=0,\qquad
h_s(x^\star)=0,\qquad
h_\ell(x^\star)<0,\ \forall \ell\neq s,
\]
where $g(x)=0$ denotes the set of equality constraints of $\mathcal P_1$, and assume that $\nabla_x g(x^\star)$ has full row rank.

If the stability constraint $h_s$ is independently binding at $x^\star$
with a strictly positive multiplier $\mu_s^\star>0$, then for any local parametrization $x=\Phi(z)$ of the manifold $\{x:g(x)=0\}$ near $x^\star$, where $z$ denotes a set of local free variables, the reduced
stationarity condition
\[
\nabla_z J_\mathrm{obj}(x^\star;\rho)
+
\mu_s^\star\,\nabla_z h_s(x^\star)
=0
\]
must hold.
\end{proposition}

\noindent\textit{Proof:}
The Lagrangian is
\[
\mathcal L
=
J_\mathrm{obj}(x;\rho)
+
\lambda^\top g(x)
+
\mu_s\,h_s(x).
\]

Since $x^\star$ is a KKT point, the stationarity condition at $x^\star$ gives
\[
\nabla_x J_\mathrm{obj}(x^\star;\rho)
+
\nabla_x g(x^\star)^\top \lambda^\star
+
\mu_s^\star \nabla_x h_s(x^\star)
=0.
\]

Since $\nabla_x g(x^\star)$
has full row rank, by the implicit function theorem, the feasible set
$\{x:g(x)=0\}$ can be locally parameterized by free variables $z$,
i.e., $x=\Phi(z)$.

By the chain rule, the reduced gradients are given by
\[
\nabla_z (\cdot)\big|_{z^\star}
=
J_\Phi(z^\star)^\top \nabla_x (\cdot)\big|_{x^\star}.
\]
where $J_\Phi(z^\star)$ is the Jacobian of $\Phi$ at $z^\star$.


\textcolor{black}{Restricting the stationarity condition to the tangent space of the feasible manifold 
$\{x : g(x)=0\}$ amounts to considering feasible directions $d$ satisfying 
$\nabla g(x^\star)\, d = 0$, which eliminates the term 
$\nabla_x g(x^\star)^\top \lambda^\star$. Hence, we obtain $d^\top \left(
\nabla_x J_\mathrm{obj}(x^\star;\rho)
+
\mu_s^\star \nabla_x h_s(x^\star)
\right)
= 0$, $\forall d \text{ with } \nabla g(x^\star)\, d = 0$. Using the local parametrization $x=\Phi(z)$, whose Jacobian spans the tangent space, this is equivalent to the reduced stationarity condition
\[
\nabla_z J_\mathrm{obj}(x^\star;\rho)
+
\mu_s^\star \nabla_z h_s(x^\star)
=0.\hspace{3.5cm} \blacksquare
\]}

Proposition~\ref{thm:necessary} characterizes when a stability constraint can induce a positive shadow price in the OPF problem. The result follows from the KKT stationarity condition, though is expressed in terms of the reduced optimality condition on the equality-constrained feasible set to highlight {how the stability constraint contributes to the reduced optimality condition and hence to the shadow price}. In the single-binding case, a necessary condition for a positive multiplier is that the reduced objective gradient is negatively aligned with the gradient of the active stability constraint. This means that the stability constraint must actively participate in the optimality condition, rather than being redundant with respect to the objective. Furthermore, since the reduced objective gradient depends on the parameterization of the objective function, this necessary condition shows that the stability constraint cannot carry a positive shadow price if this condition is not satisfied.

\begin{proposition}[Sufficient condition for positive shadow prices of stability constraints in OPF]
\label{thm:sufficient}
Under the same setting as Proposition~\ref{thm:necessary}, consider a feasible point $x$ at which a single stability constraint $h_s$ is active, i.e., $g(x)=0$,
$h_s(x)=0$,
$h_\ell(x)<0$, $\forall \ell\neq s$, where $g(x)=0$ denotes the equality constraints of $\mathcal P_1$. Suppose there
exist a parameter $\rho$ and a scalar $\mu_s>0$ such that
\[
\nabla_z J_\mathrm{obj}(x;\rho)
+
\mu_s\,\nabla_z h_s(x)
=0.
\]

If, in addition, the second-order sufficient condition (SOSC) holds for
the reduced problem, then $x$ is a strict local minimizer of
$\mathcal P_1$, and the constraint $h_s$ is independently binding at
$x$ with positive shadow price $\mu_s>0$.
\end{proposition}

\noindent\textit{Proof:}
Since $\nabla_x g(x)$ has full row rank, the equality-constrained feasible set $\{x:g(x)=0\}$ can be locally parameterized as $x=\Phi(z)$ around $x$, where $z$ denotes a set of local free variables.

Since all other inequality constraints are strictly inactive at
$x$, the OPF problem locally reduces to
\[
\min_z\ J_\mathrm{obj}(\Phi(z);\rho)
\quad \mathrm{s.t.}\quad
h_s(\Phi(z))\le 0.
\]

By assumption, the reduced first-order stationarity condition holds with
$\mu_s>0$, and the reduced SOSC ensures that $z$ is a strict
local minimizer of the reduced problem.

Therefore, $x=\Phi(z)$ is a strict local minimizer of
$\mathcal P_1$. Since $h_s(x)=0$ and $\mu_s>0$, the
constraint $h_s$ is independently binding at $x$ with positive
shadow price.
\hfill $\blacksquare$

Proposition~\ref{thm:sufficient} provides a constructive counterpart to Proposition~\ref{thm:necessary}, showing that positive shadow prices arise when this sufficient condition is satisfied. Together with Theorem~\ref{thm:zero_shadow}, these results rely on the existence of an operating point where a stability constraint is independently binding. In practice, identifying such points is nontrivial, especially in high-dimensional OPF problems where verifying SOSC is difficult. A useful perspective is provided by the inward feasible descent principle, which states that a constraint affects optimality only if it eliminates feasible descent directions. Following this insight, we can construct such operating points by restricting variables to eliminate descent directions, thereby obtaining locally optimal points with positive shadow prices.

\begin{theorem}[Positive shadow price of stability constraint with reactive power costs]
\label{thm:qcost}
Under the same setting as Proposition~\ref{thm:sufficient}, consider the
quadratic objective function \eqref{eq:general_cost} in OPF, and define coefficients vector
\[
\rho := (b,c,d)\in\mathbb{R}^{3|\mathcal G|}.
\]

If there exist coefficients $\rho$ with at least one $d_i>0$ and a scalar
$\mu_s^\star>0$ such that
\begin{equation}
A(x^\star)\,\rho+\mu_s^\star\,\nabla_z h_s(x^\star)=0,
\label{eq:qcost_realization}
\end{equation}
where $A(x^\star)$ denotes the matrix mapping the objective parameters
$\rho$ to the reduced gradient of $J_\mathrm{obj}$ on the
equality-constrained feasible set, i.e.,
\[
\nabla_z J_\mathrm{obj}(x^\star;\rho)=A(x^\star)\,\rho,
\]
and if the reduced SOSC in
Proposition~\ref{thm:sufficient} holds, then $x^\star$ is a strict local
minimizer of $\mathcal P_1$, and the stability constraint $h_s$ is
independently binding at $x^\star$ with positive shadow price
$\mu_s^\star>0$.
\end{theorem}

\noindent\textit{Proof:}
Under the quadratic objective \eqref{eq:general_cost}, the objective
gradient depends only on the coefficients $(b,c,d)$, since the constants
$a_i$ vanish under differentiation. In particular, at $x^\star$ we have
\[
\nabla_{(P_G,Q_G)} J_\mathrm{obj}(x^\star;\rho)
=
D(x^\star)\,\rho,
\]
where
\[
D(x^\star)=
\begin{pmatrix}
I & 2\,\mathrm{diag}(P_G^\star) & 0\\
0 & 0 & 2\,\mathrm{diag}(Q_G^\star)
\end{pmatrix}.
\]

Let $x=\Phi(z)$ be a local parametrization of the equality-constrained
feasible set $\{x:g(x)=0\}$ around $x^\star$, where $z$ denotes a set of local free variables, and let
\[
J_\Phi(z^\star):=\frac{\partial \Phi}{\partial z}(z^\star)
\]
denote the Jacobian of $\Phi$ at $z^\star$. By the chain rule, the
reduced gradient of the objective is
\[
\nabla_z J_\mathrm{obj}(x^\star;\rho)
=
J_\Phi(z^\star)^\top D(x^\star)\,\rho.
\]
Hence,
\[
A(x^\star):=J_\Phi(z^\star)^\top D(x^\star)
\]
is exactly the matrix mapping $\rho$ to the reduced gradient of
$J_\mathrm{obj}$, so that
\[
\nabla_z J_\mathrm{obj}(x^\star;\rho)=A(x^\star)\,\rho.
\]

Therefore, condition \eqref{eq:qcost_realization} coincides with the
reduced first-order stationarity condition
\[
\nabla_z J_\mathrm{obj}(x^\star;\rho)
+\mu_s^\star\,\nabla_z h_s(x^\star)=0
\]
required in Proposition~\ref{thm:sufficient}. Since, in addition, the reduced
SOSC holds, Proposition~\ref{thm:sufficient}
implies that $x^\star$ is a strict local minimizer of $\mathcal P_1$ and
that the stability constraint $h_s$ is independently binding at $x^\star$
with positive shadow price $\mu_s^\star>0$.

\hfill $\blacksquare$

Theorem~\ref{thm:qcost} shows that incorporating reactive-power costs
expands the set of achievable (reduced) objective gradients through the
parameter $\rho$. In contrast to $P$-only objectives, where the reduced
gradient is structurally restricted and may vanish along feasible
directions, the inclusion of reactive-power terms introduces additional
degrees of freedom that allow alignment between the reduced objective
gradient and the stability constraint gradient. As a result, the stationarity condition required for a positive shadow
price can be satisfied, making independently binding stability constraints
with nonzero economic value attainable. The practical importance of the theoretical results proved in this section will be discussed in Section~\ref{sec:numerical}.

\section{Numerical Studies}\label{sec:numerical}

\subsection{Quality of the Decentralized Stability Criteria}

To evaluate the quality of the decentralized stability criteria, we compare it against the exact small-signal stability criteria based on the eigenvalues of the linearized system matrix. Since the decentralized condition is designed as a tractable sufficient condition, the main question is how conservative it is relative to the eigenvalue-based benchmark.

For this purpose, we consider a grid of operating points, and define the \emph{gap ratio} as the fraction of operating points that are certified stable by the eigenvalue criterion but rejected by the decentralized stability criterion,
\[
\xi
:=
\frac{
\card\!\left(\mathcal{S}_{\mathrm{eig}} \setminus \mathcal{S}_{\mathrm{dec}}\right)
}{
\card\!\left(\mathcal{S}_{\mathrm{eig}}\right)
},
\]
where $\mathcal{S}_{\mathrm{eig}}$ denotes the set of operating points for which the reduced linearized system is small-signal stable, i.e., all eigenvalues except the trivial zero mode have strictly negative real parts, and $\mathcal{S}_{\mathrm{dec}}$ denotes the set of operating points that satisfy the decentralized stability condition. 
A smaller value of $\xi$ indicates a less conservative decentralized criterion.

We perform the test on the two-bus system by fixing the active-power loop parameters and varying the reactive droop coefficients as well as the network susceptance parameter. All parameters are in p.u. unless otherwise specified. Specifically, the active-power loop parameters are fixed at
$m^p = 6$, $\beta_p = 1$, $\tau_p = 0.1$,
and the remaining reactive-loop parameters are fixed to $\beta_q = 1$, $\tau_q = 0.1$. The scanned parameters are
$B \in \{2,4,6,8\}$ p.u., and the grid for $m_1^q$ and $m_2^q$ contains $9$ uniformly spaced points in $[1,5]$. For each parameter tuple $(B,m_1^q,m_2^q)$, we further sample the operating-point space over $V_1,V_2 \in [0.95,1.05]$, $\theta_2 \in [-0.525,\,0.525]$, using $31$ grid points for $V_1$ and $V_2$, and $61$ grid points for $\theta_2$. Hence, for each fixed $(B,m_1^q,m_2^q)$, the gap ratio is computed from $31 \times 31 \times 61$ sample operating points.


\begin{figure}[!t]
\centering
\subfloat[$B=2$ p.u.]{\includegraphics[width = 4.4cm]{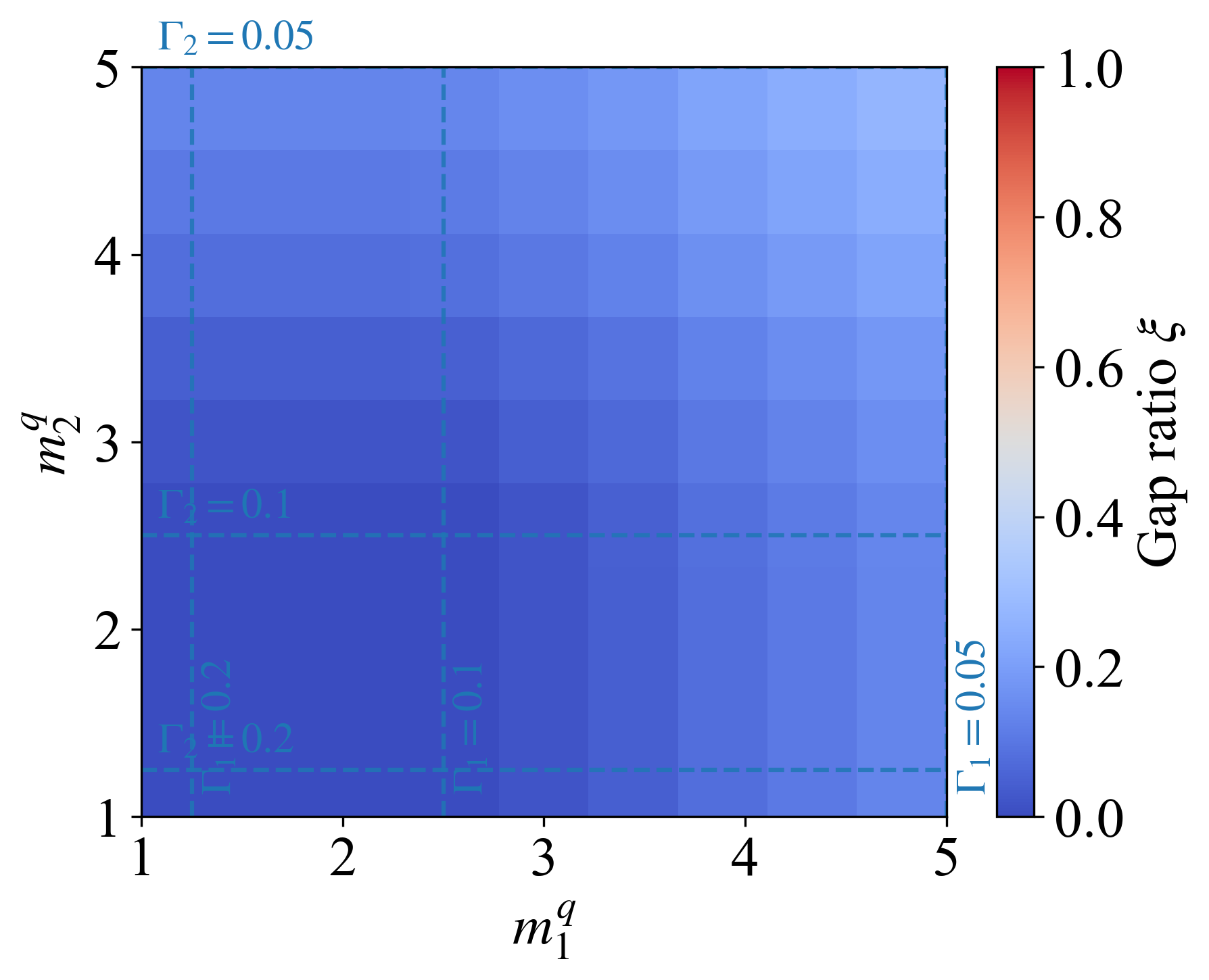}\label{figsub:1}}
\subfloat[$B=4$ p.u.]{\includegraphics[width = 4.4cm]{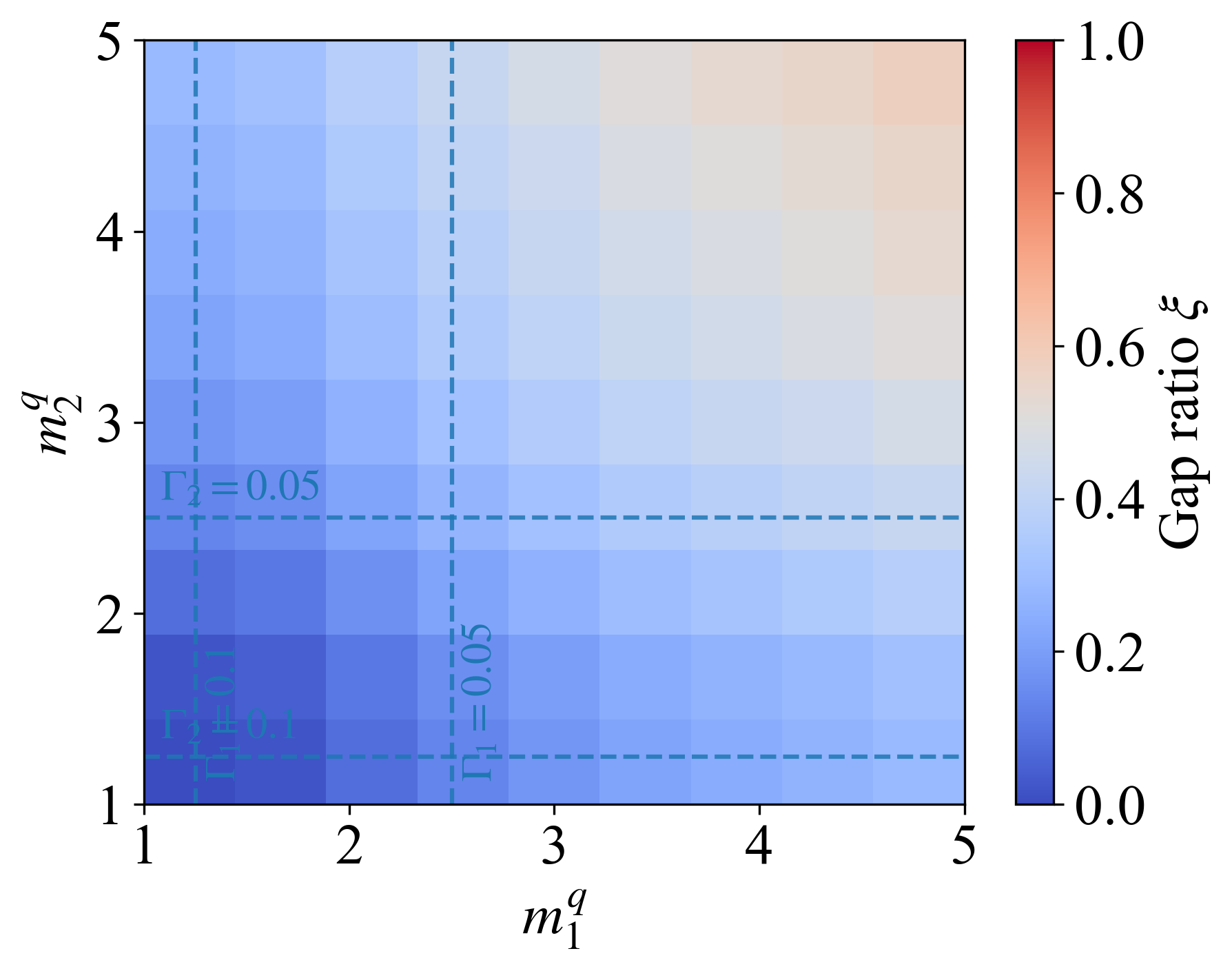}\label{figsub:3}}\hfil
\subfloat[$B=6$ p.u.]{\includegraphics[width = 4.4cm]{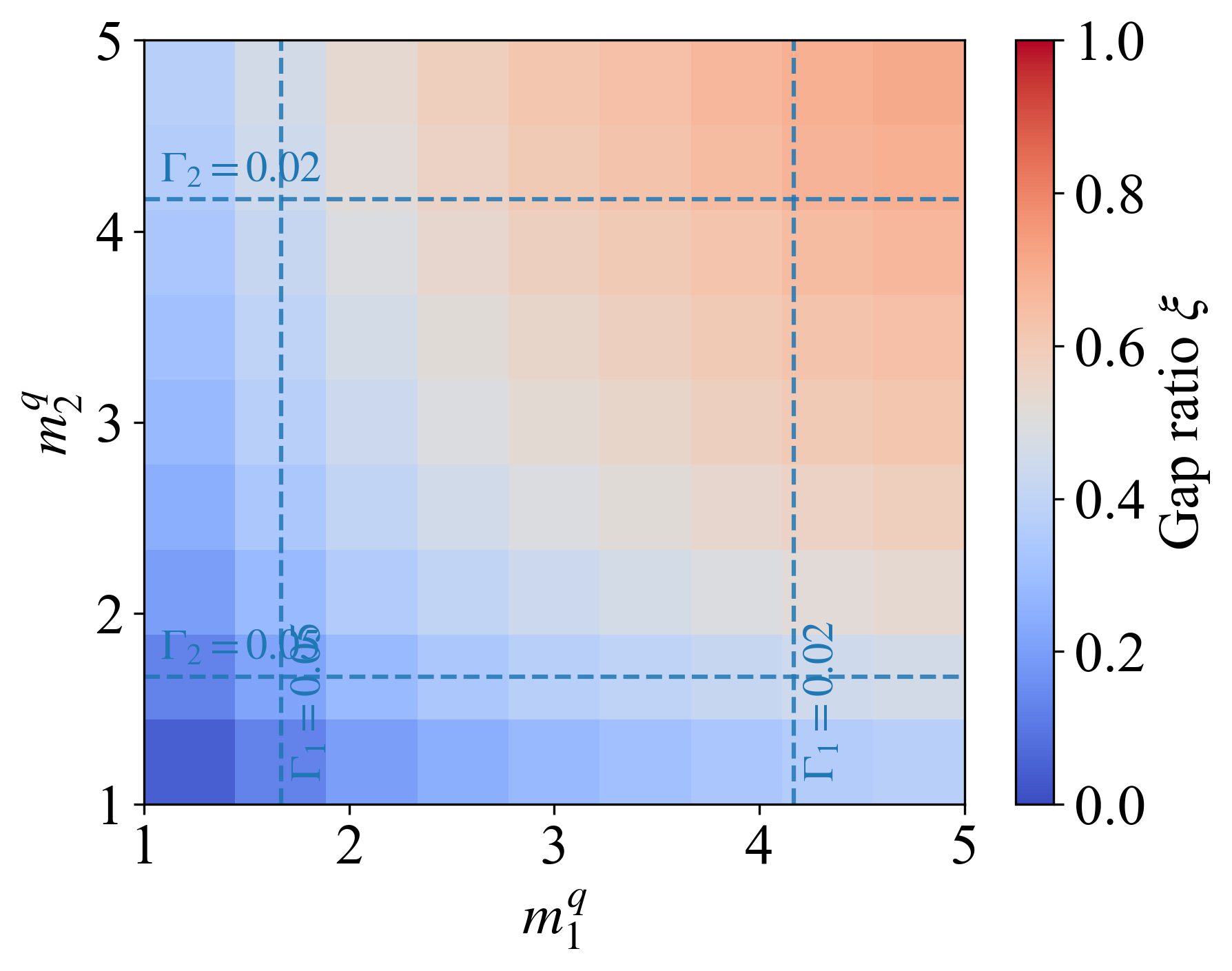}\label{figsub:4}}
\subfloat[$B=8$ p.u.]{\includegraphics[width = 4.4cm]{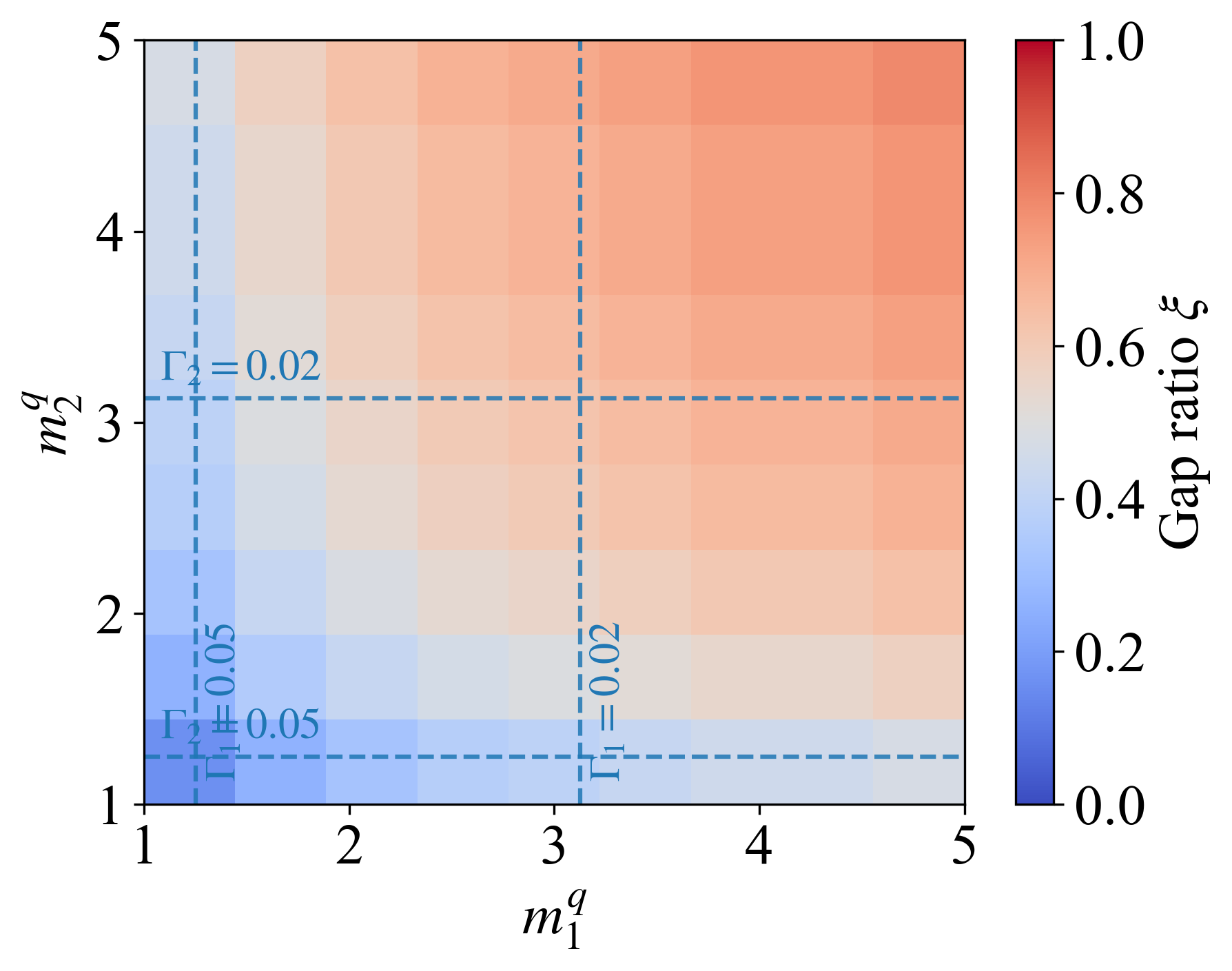}\label{figsub:5}}
\caption{Heatmap of the gap ratio $\xi$ over the space of droop parameters $m_1^q, m_2^q$ for different values of $B$.}
\label{fig5}
\end{figure}

Figure~\ref{fig5} shows the heatmap of the gap ratio $\xi$ over the droop parameter space $(m_1^q,m_2^q)$ for different values of $B$. Several clear trends can be observed. First, for each fixed value of $B$, the gap ratio $\xi$ increases monotonically with both $m_1^q$ and $m_2^q$. In particular, the lower-left region (small droop coefficients) exhibits near-zero gap ratio, indicating that the decentralized condition closely matches the eigenvalue-based stability region. In contrast, as $m_1^q$ and $m_2^q$ increase, the gap ratio gradually approaches one, meaning that a large fraction of stable operating points does not pass the decentralized criterion. This behavior shows that the proposed condition becomes more conservative for larger droop coefficients. Second, the conservativeness of the decentralized criterion also depends strongly on the network parameter $B$. As $B$ increases, the overall level of $\xi$ rises significantly across the entire parameter space. 
These trends can be explained by the structure of the decentralized stability bounds $\Gamma_i = \frac{1}{2 m_i^q \beta_q B}$, where increasing either the droop coefficient $m_i^q$ or the network parameter $B$ reduces the admissible voltage-difference bound on $(V_1,V_2)$. 
In summary, the results demonstrate that the proposed decentralized stability criterion is tight for small droop gains and weak coupling, but becomes increasingly conservative as the droop gains or network strength increase.

\begin{figure}[!t]
\centering
\subfloat[$B=2$ p.u.]{\includegraphics[width = 4.4cm]{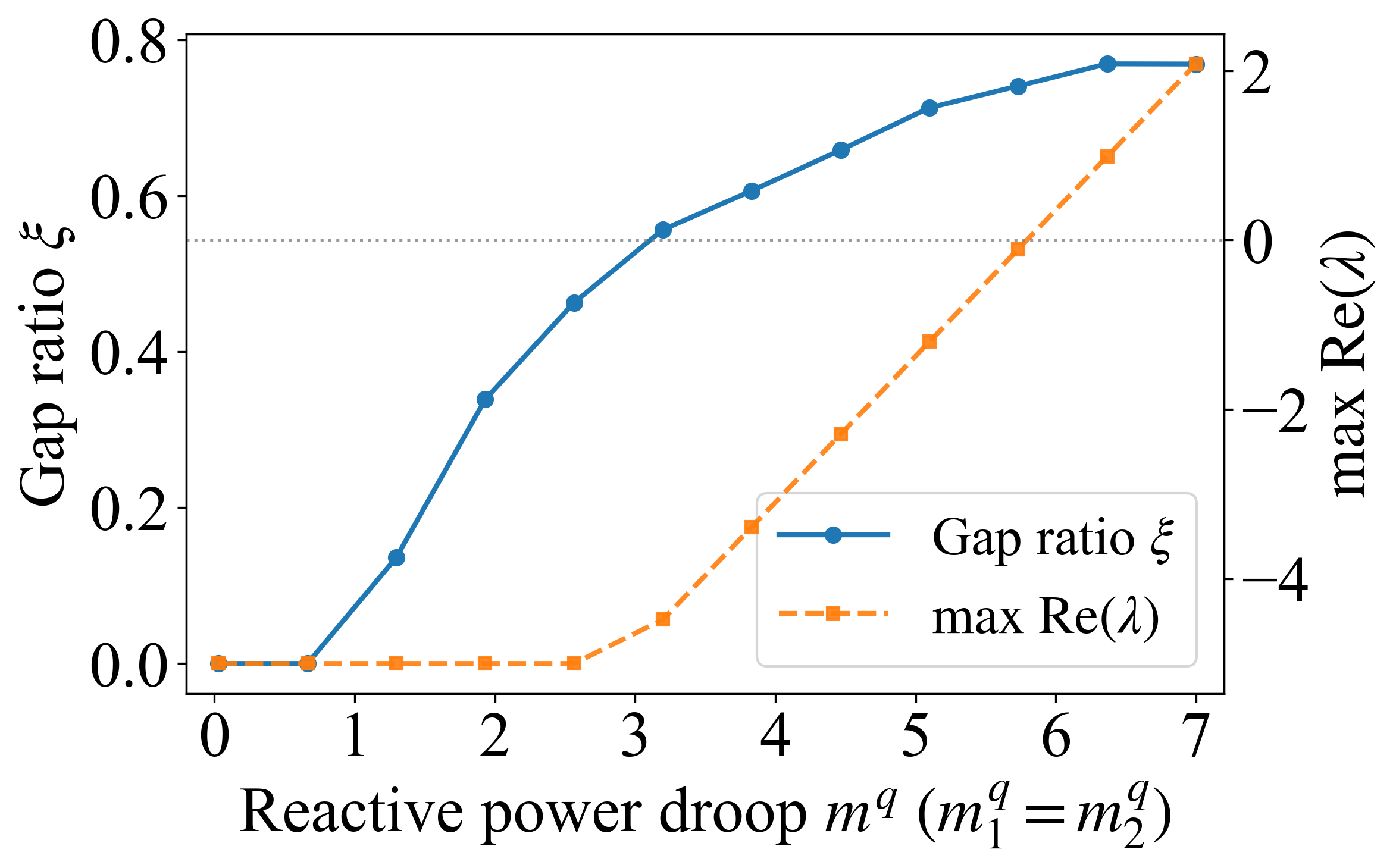}\label{figsub:1}}
\subfloat[$B=4$ p.u.]{\includegraphics[width = 4.4cm]{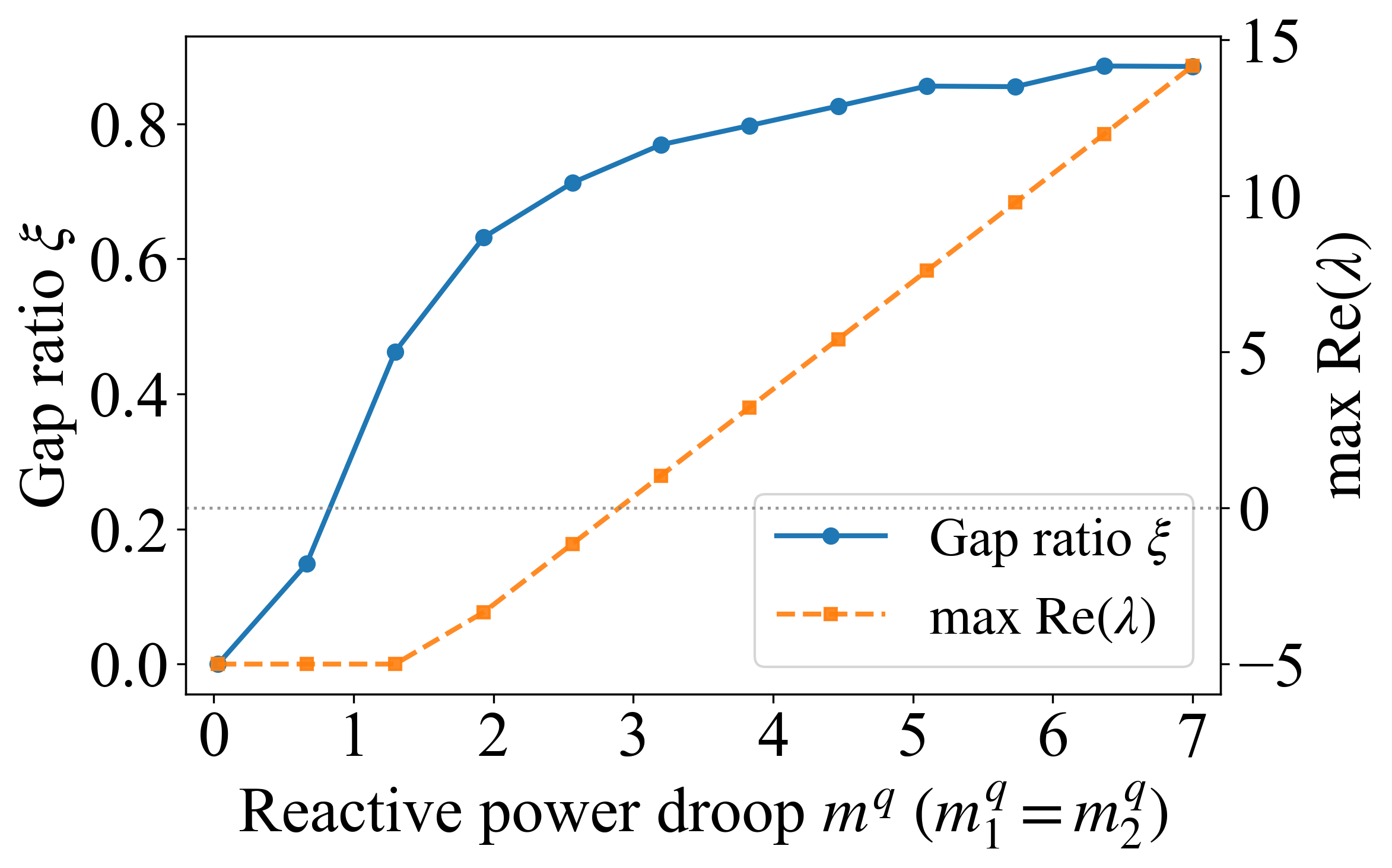}\label{figsub:3}}\hfil
\subfloat[$B=6$ p.u.]{\includegraphics[width = 4.4cm]{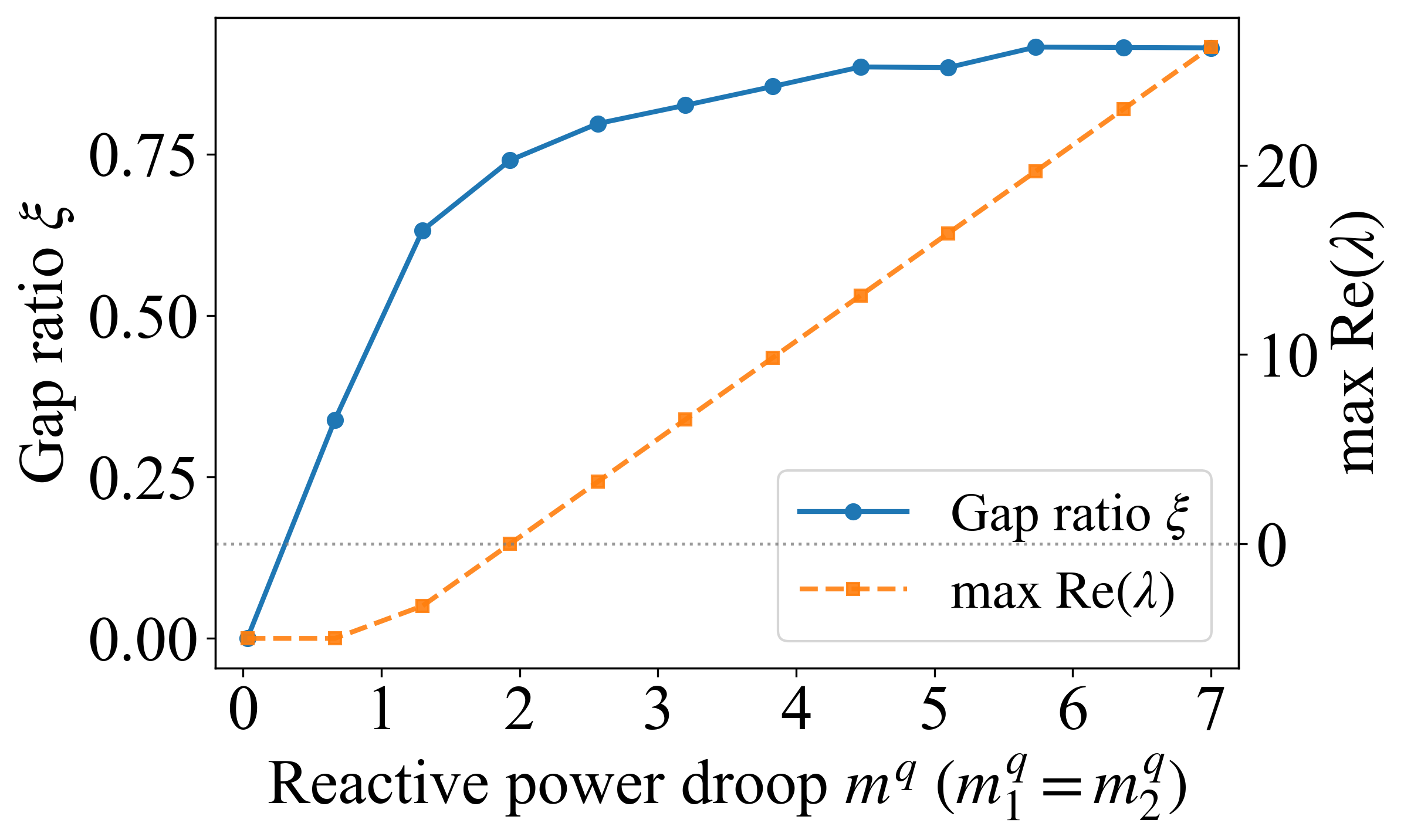}\label{figsub:4}}
\subfloat[$B=8$ p.u.]{\includegraphics[width = 4.4cm]{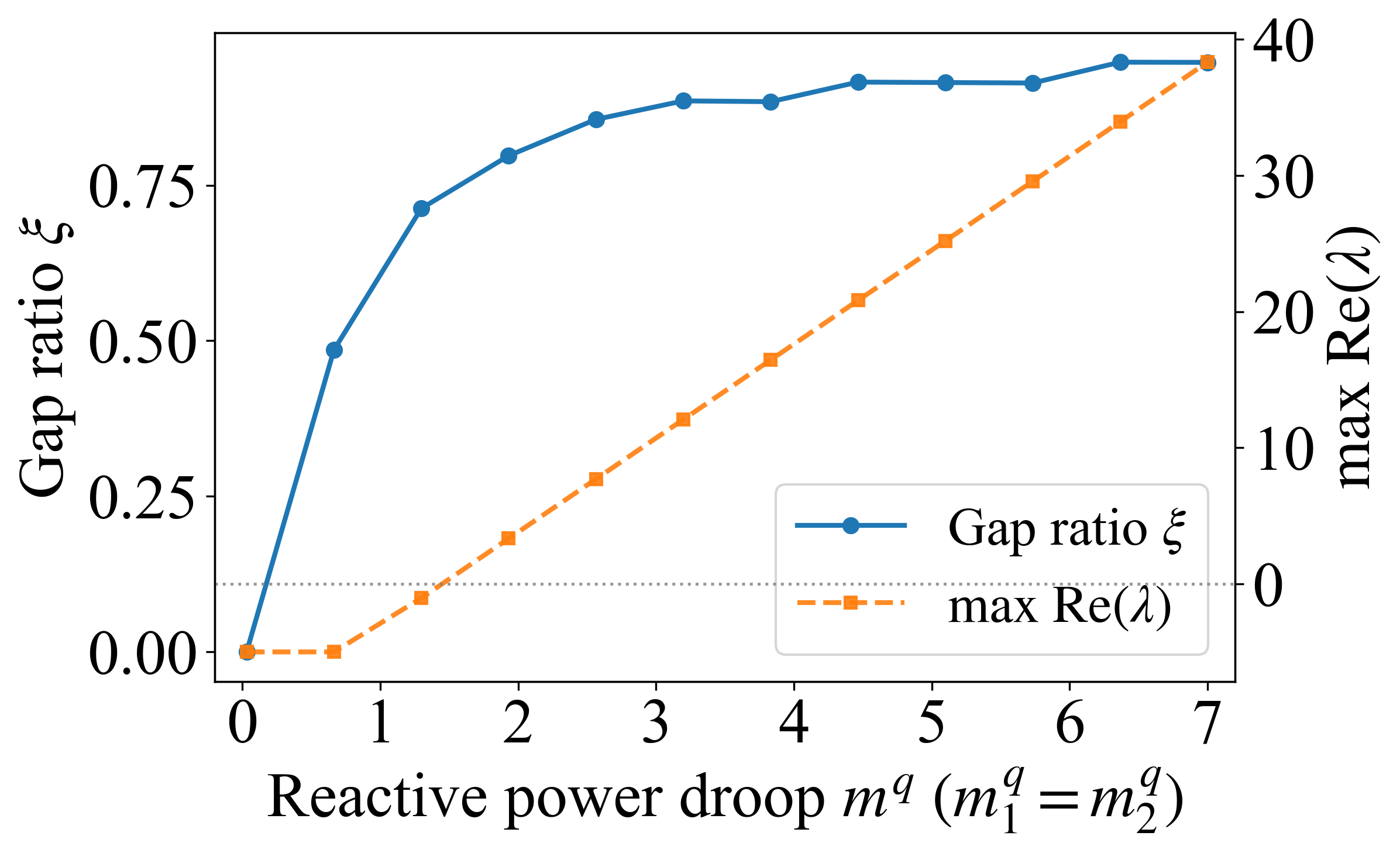}\label{figsub:5}}
\caption{Gap ratio $\xi$ and the eigenvalue-based stability margin $\max \operatorname{Re}(\lambda)$ along $m_1^q = m_2^q$ for different values of $B$.}
\label{fig6}
\end{figure}



Figure~\ref{fig6} further compares the gap ratio $\xi$ and the eigenvalue-based stability margin $\max \operatorname{Re}(\lambda)$ (the trivial zero eigenvalue caused by angle invariance is removed) along the symmetric slice $m_1^q = m_2^q = m^q$ with $m^q \in [0.03,\, 7]$ for different values of $B$. All parameters are kept the same as in the previous experiment except the operating-point scan is enlarged to a $61 \times 61 \times 61$ grid for $V_1,V_2 \in [0.85,1.15]$ and $\theta_2 \in [-0.525,\,0.525]$.
Consistent with the heatmap results, the gap ratio exhibits an overall increasing trend with respect to $m^q$ and gradually saturates close to one. 

While the decentralized admissible set shrinks monotonically with $m^q$, the eigenvalue-stable set also changes with $m^q$ due to the variation of the system eigenvalues.
In particular, operating points near the stability boundary may switch their classification as $m^q$ varies, leading to small local decreases of $\xi$. 
The variation of $\max \operatorname{Re}(\lambda)$ with respect to $m^q$ is governed primarily by a single dominant mode associated with the reactive power-voltage dynamics. 
After reordering the states, the Jacobian can be decomposed into an active power-angle block and a reactive power-voltage block. The active power-angle subsystem contains one trivial zero eigenvalue due to angle invariance, and its remaining modes are independent of $m^q$, with real parts determined by $-1/\tau_p$ and thus strictly negative. Similarly, part of the reactive power dynamics contributes modes with real parts dominated by $-1/\tau_q$, which are also insensitive to $m^q$. As a result, the dependence on $m^q$ enters only through a single reactive power-voltage mode, whose real part increases with $m^q$ and eventually determines $\max \operatorname{Re}(\lambda)$. 

This structural insight also explains the different behaviors observed for small and large values of $B$. For small $B$ (e.g., $B=2$), the worst-case eigenvalue initially corresponds to an $m^q$-insensitive mode (e.g., one with real part close to $-1/\tau_p$ or $-1/\tau_q$), resulting in a nearly flat region of $\max \operatorname{Re}(\lambda)$ for small $m^q$. As $m^q$ increases, the reactive power-voltage mode becomes dominant, leading to a transition in the identity of the worst-case eigenvalue and a subsequent increase in $\max \operatorname{Re}(\lambda)$. In contrast, for larger values of $B$, the reactive power-voltage mode is already the least stable mode even at small $m^q$, so that $\max \operatorname{Re}(\lambda)$ increases monotonically with $m^q$ from the outset.

\subsection{Illustration of Shadow Price for Stability Constraints}

\subsubsection{2-bus system}
We first consider a 2-bus example to illustrate the theoretical results developed in this paper.
In this system, the stability constraints are expressed as,
\begin{align}
    & \mathrm{stab1}: \ V_{2}-V_{1}\leq \Gamma_1,\nonumber\\
& \mathrm{stab2}: \ V_{1}-V_{2}\leq \Gamma_2,\nonumber
\end{align}
where $\Gamma_1 = \frac{1}{2m^q_1\beta^q_1B}$, and $\Gamma_2 = \frac{1}{2m^q_2\beta^q_2B}$. 

In the following, we investigate the role of the voltage-difference stability constraints. For clarity, we parameterize these constraints directly in terms of $\Gamma_1$ and $\Gamma_2$, rather than the droop coefficients $m_i^q$, since $\Gamma_i$ explicitly determines the admissible voltage deviation and thus the effective size of the feasible region.

First, we consider the setting where the objective depends only on active power, in which independently binding stability constraints arise but exhibit zero shadow price, as proved in Theorem~\ref{thm:zero_shadow}.

The parameters are chosen as
$B=10$, $\overline P=2.5$, $\overline Q=2$, $\underline V=0.95$, $\overline V=1.05$, $\overline S=3$, with the reference-bus voltage fixed at $V_1=1$, generation cost coefficients
$(a_1,b_1,c_1)=(0,0.55,0.12)$, $(a_2,b_2,c_2)=(0,0.6,0.16)$, and load values $(P_{D,1},P_{D,2})=(0.45,0.7)$, $(Q_{D,1},Q_{D,2})=(0.45,0)$. We scan the upper voltage-difference limit over $\Gamma_1\in\{0.005,0.01,0.015,0.02,0.025,0.03\}$, and focus on the decentralized constraint $\mathrm{stab1}: \ V_2-V_1\le \Gamma_1$. The resulting nonlinear optimization problems are solved using a sequential quadratic programming (SQP) method with multiple initializations.

\begin{table}[t]
\centering
\caption{Independent binding decentralized stability constraint with zero shadow price under active-power-only cost}
\label{tab:zero_shadow}
\begin{tabular}{c c c c c c}
\toprule
$\Gamma_1$ & $J_{\mathrm{obj}}$ & $\theta_2$ & $V_2$ & Independent binding & $\lambda^{\mathrm{stab}}_{12}$ \\
\midrule
0.005 & 0.7456 & -0.02950 & 1.005 & Yes & 0.0 \\
0.010 & 0.7456 & -0.02935 & 1.010 & Yes & 0.0 \\
0.015 & 0.7456 & -0.02921 & 1.015 & Yes & 0.0 \\
0.020 & 0.7456 & -0.02907 & 1.020 & Yes & 0.0 \\
0.025 & 0.7456 & -0.02892 & 1.025 & Yes & 0.0 \\
0.030 & 0.7456 & -0.02878 & 1.030 & Yes & 0.0 \\
\bottomrule
\end{tabular}
\end{table}

Table~\ref{tab:zero_shadow} numerically verifies Theorem~\ref{thm:zero_shadow} for the active-power-only objective case. For all tested values of $\Gamma_1$, the boundary solution is independently supported by $\mathrm{stab1}$, that is, the active set is $\{\mathrm{stab1}\}$ in every case. Meanwhile, although $(\theta_2,V_2)$ varies with $\Gamma_1$, the objective value remains unchanged, indicating that these solutions lie on a flat optimal manifold. Moreover, we numerically verify that each boundary solution is a local minimizer. The shadow prices $\lambda_{12}^{\mathrm{stab}}$ are zero for all cases, confirming that the stability constraint, although independently binding, does not impose any marginal economic cost.

Second, we incorporate reactive power costs into the objective function, which can admit independently binding stability constraints with positive shadow prices, as proved in Theorem~\ref{thm:qcost}. 

The parameters are chosen as $B=8$, $\overline P=2.5$, $\overline Q=2$, $\underline V=0.95$, $\overline V=1.05$, $\overline S=3$, with the reference-bus voltage fixed at $V_1=1$. The cost coefficients are given by $(a_1,b_1,c_1,d_1)=(0,0.2,0.05,8)$, $(a_2,b_2,c_2,d_2)=(0,0.8,0.4,3)$.
The nominal load values are $(P_{D,1},P_{D,2})=(0.15,0.95)$, $(Q_{D,1},Q_{D,2})=(0.45,0.05)$, and are slightly perturbed to identify representative operating points.

\begin{table}[t]
\centering
\caption{Independent binding with positive shadow price under reactive-power cost}
\label{tab:positive_shadow}
\begin{tabular}{ccccccc}
\toprule
$\Gamma_1$ & $J_{\mathrm{obj}}$ & $\theta_2$ & $V_2$ & Independent binding & $\lambda_{12}^{\mathrm{stab}}$ \\
\midrule
0.040 & 1.4009 & $-0.0951$ & 1.0400 & Yes & 8.0483 \\
0.042 & 1.3878 & $-0.0959$ & 1.0420 & Yes & 5.0701 \\
--    & 1.3792 & $-0.0973$ & 1.0454 & No  & --     \\
\bottomrule
\end{tabular}
\end{table}

Table~\ref{tab:positive_shadow} reports three operating points corresponding to different treatments of the stability constraint $\mathrm{stab1}$. The first row is a baseline case where the constraint is active, the second row relaxes $\Gamma_1$, and the third row removes the constraint entirely. In the first two rows, $\mathrm{stab1}$ is the only active inequality constraint and admits a strictly positive shadow price. This shows that, by appropriately chosen reactive-power cost coefficients while keeping the active-power cost fixed, one can realize an operating point consistent with the first-order structure characterized in Proposition~\ref{thm:necessary}. Moreover, when local optimality holds and all other inequality constraints remain inactive, Proposition~\ref{thm:sufficient} further guarantees that the stability constraint is independently binding with $\lambda_{12}^{\mathrm{stab}}>0$. Relaxing $\Gamma_1$ decreases the objective, and removing the constraint yields a further decrease, confirming that the stability constraint imposes a nonzero marginal cost.


\subsubsection{IEEE 39-bus system}

We consider the IEEE 39-bus system, which consists of 10 generator buses. The voltage base is 345 kV and power base is 100 MVA. Bus 31 is chosen as the reference bus, and its voltage angle is fixed to zero. The generation cost coefficients are taken from the standard test case and converted to the per-unit setting. 
All branch resistances and shunt components are set to zero. Accordingly, the susceptance matrix is given by $B=-\operatorname{Im}(Y_{\mathrm{bus}})$.

The problem is formulated following $\mathcal P_1$, where the decentralized 
stability constraints in \eqref{eq:pairwise} are incorporated. For constructing the stability constraints, Kron reduction is performed by eliminating all non-generator buses, yielding the reduced admittance matrix \(Y_{\mathrm{red}}\)\footnote{To construct the reduced network, constant-power loads are first converted into equivalent shunt admittances under nominal bus voltage \(V=1\) p.u, and added to the diagonal of the bus admittance matrix.} and the corresponding reduced susceptance matrix
$B^{\mathrm{red}}=-\operatorname{Im}(Y_{\mathrm{red}})$. This renders a 10-bus reduced network. Moreover, in the reduced network, 
all retained generator buses are pairwise adjacent, and the stability 
constraints are imposed for all ordered pairs of distinct generator buses. 
As a result, a total of $10\times 9=90$ voltage-difference stability 
constraints are included.

We first consider the $P$-only objective in \eqref{eq:general_cost} by setting $d_i=0$. The resulting nonlinear optimization problem is implemented in Pyomo and solved using the IPOPT solver. Tight solver tolerances are used, with a convergence tolerance of $10^{-8}$ and an acceptable tolerance of $10^{-6}$. To improve convergence, a warm-start strategy is adopted, where each run is initialized from the solution of the previous parameter setting. 

\begin{figure}[!t]
\centering

\begin{minipage}{0.48\linewidth}
\centering
\begin{overpic}[width=\linewidth]{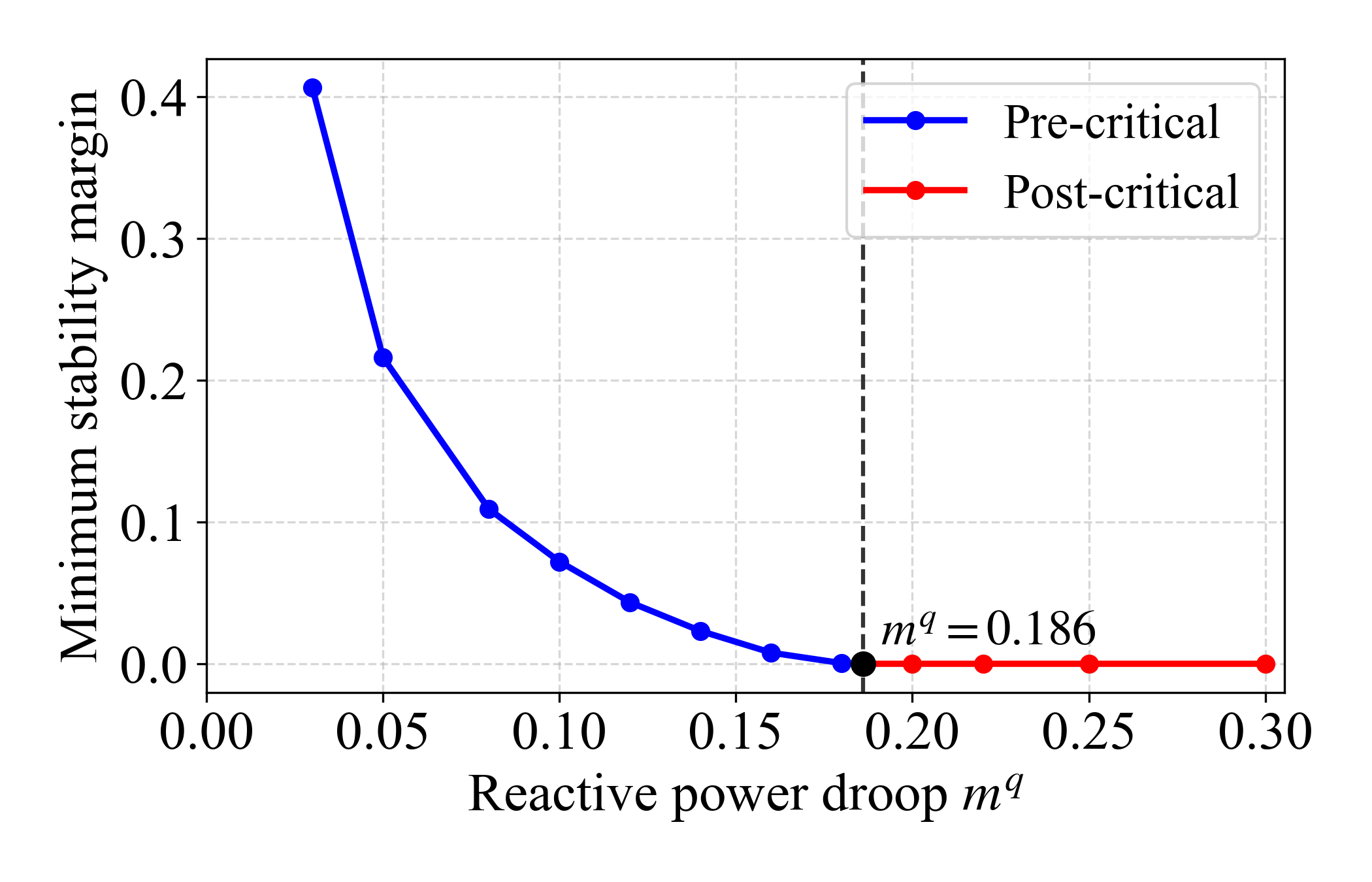}
    \put(0,5.5){\footnotesize (a)}
\end{overpic}
\end{minipage}
\hfill
\begin{minipage}{0.48\linewidth}
\centering
\begin{overpic}[width=\linewidth]{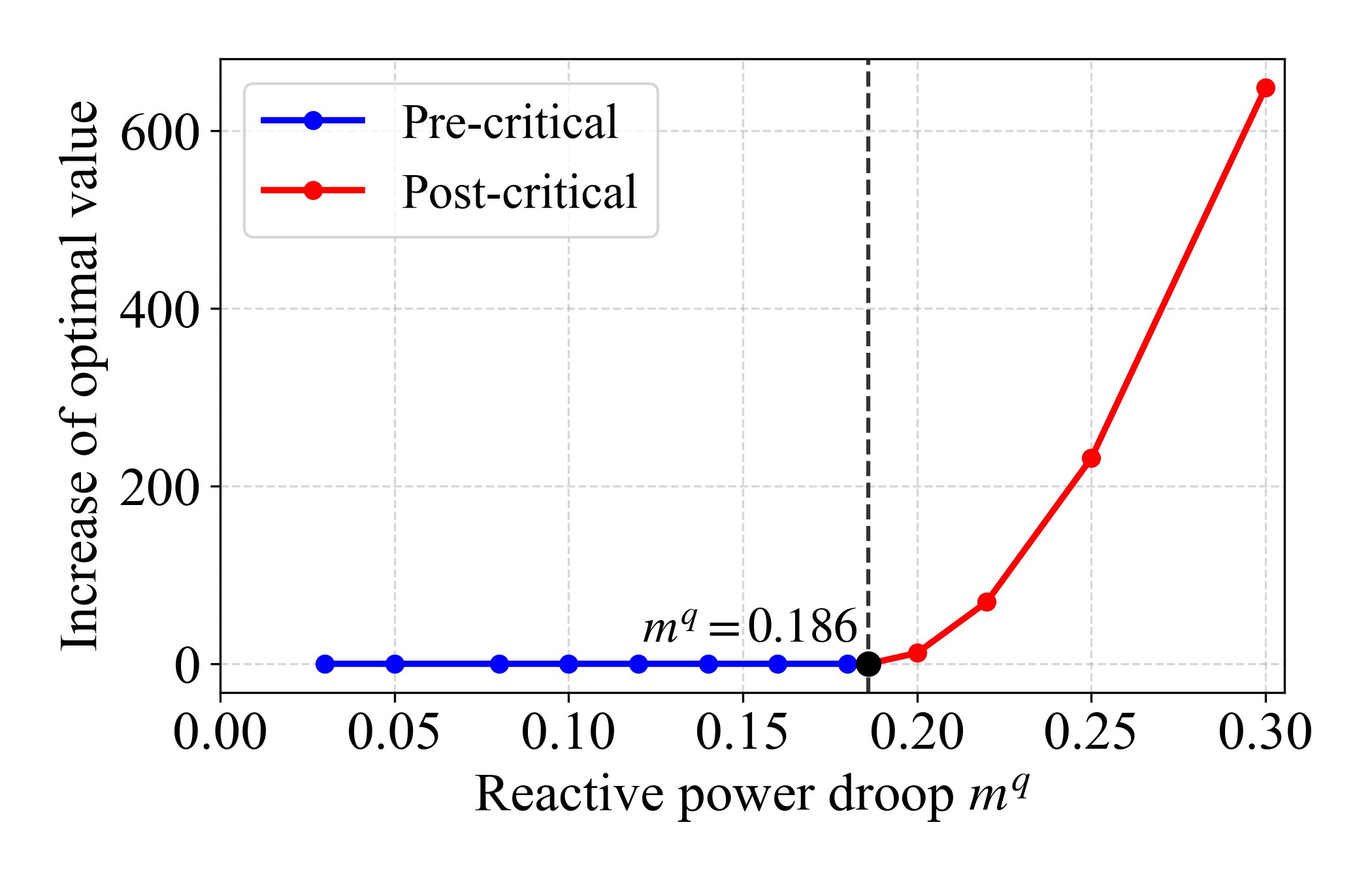}
    \put(0,5.5){\footnotesize (b)}
\end{overpic}
\end{minipage}

\caption{Effects of reactive power droop $m^q$ on: (a) Minimum stability margin, and (b) Increase of optimal value (compared to without stability constraints) in the IEEE 39-bus system under a $P$-only objective.}
\label{fig7}

\end{figure}

Figure~\ref{fig7} illustrates the impact of the reactive power droop parameter $m^q$ (uniform for all inverters) on the minimum stability margin and optimal value of the stability-constrained OPF. The minimum stability margin is defined as the smallest slack among all stability constraints, representing the closest distance to the stability boundary. As shown in Fig.~\ref{fig7}(a), the minimum stability margin decreases monotonically as $m^q$ increases, indicating that the stability constraints become progressively tighter. At a critical value $m^q \approx 0.186$, the minimum margin reaches zero, which indicates that at least one stability constraint becomes active. Fig.~\ref{fig7}(b) shows the corresponding increase in the optimal value relative to the baseline OPF without stability constraints. For $m^q$ below the critical value, the optimal value remains unchanged, indicating that the stability constraints are inactive and do not affect the optimal solution. 
For $m^q$ beyond the critical point, the optimal value increases rapidly, indicating that the stability constraints start to restrict the feasible set and force a deviation from the baseline optimal solution.

\begin{figure}[!t]
\centering

\begin{minipage}{0.48\linewidth}
\centering
\begin{overpic}[width=\linewidth]{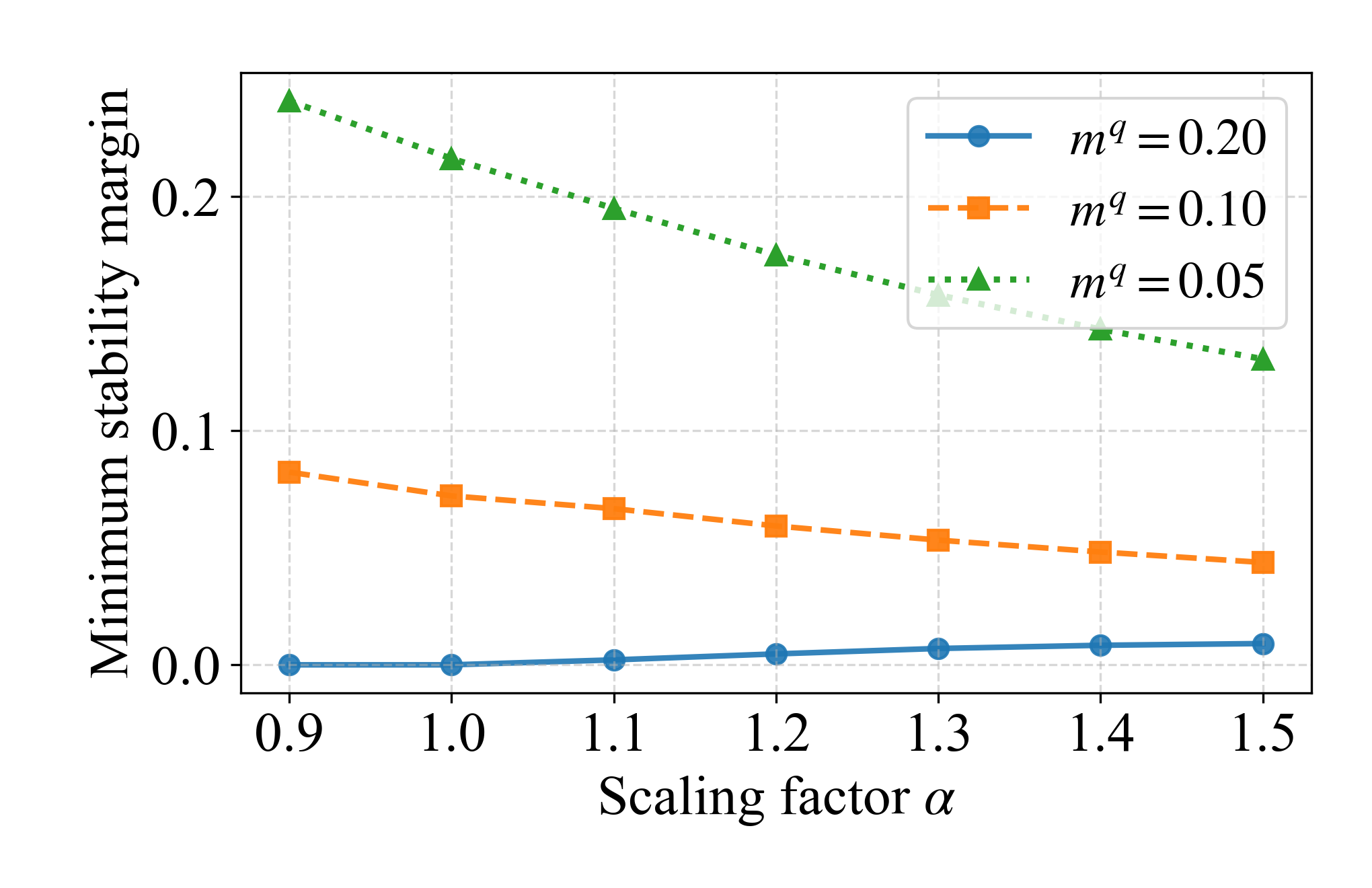}
    \put(0,5.5){\footnotesize (a)}
\end{overpic}
\end{minipage}
\hfill
\begin{minipage}{0.48\linewidth}
\centering
\begin{overpic}[width=\linewidth]{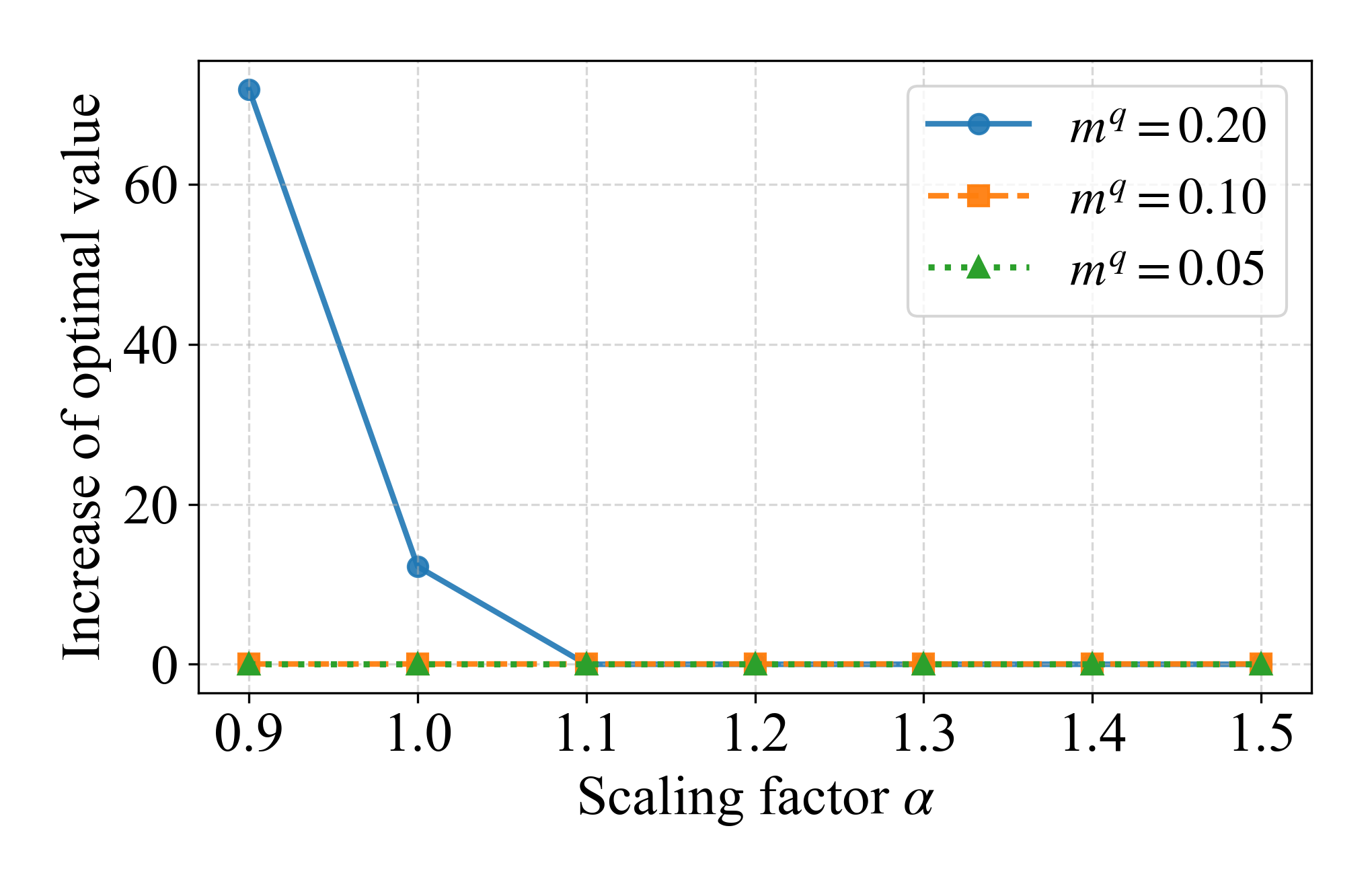}
    \put(0,5.5){\footnotesize (b)}
\end{overpic}
\end{minipage}

\caption{Effects of scaling factor $\alpha$ for network 
susceptance matrix on: (a) Minimum stability margin, and (b) Increase of optimal value (compared to without stability constraints) in the IEEE 39-bus system under a $P$-only objective.}
\label{fig8}

\end{figure}

Figure~\ref{fig8} examines the effect of the scaling factor $\alpha$ on network susceptance matrix for three values of $m^q$. The behaviors are qualitatively different on the two sides of the critical value $m^q \approx 0.186$ identified in Fig.~\ref{fig7}. For $m^q=0.05$ and $0.10$, which are below the critical value, the minimum stability margin decreases as $\alpha$ increases; see Fig.~\ref{fig8}(a). The stability constraints remain inactive, as also confirmed by the zero optimal value increase in Fig.~\ref{fig8}(b). Hence, the optimal solution is essentially the baseline OPF solution, and the observed reduction in the minimum margin is mainly due to the tightening of the stability bound $\frac{1}{2m_i^q|B_{ii}^{\mathrm{red}}|}$ as $\alpha$ increases. In contrast, for $m^q=0.2$, which is above the critical value, the trend is reversed: the minimum stability margin increases with $\alpha$, and the objective increase drops to zero once $\alpha$ becomes sufficiently large. This indicates that, the stability constraints are already influencing the OPF solution for small $\alpha$, and the re-optimized operating point evolves with $\alpha$ in a way that reduces the critical voltage differences. As a result, the reduction in the left-hand side of the active constraint dominates the shrinkage of the right-hand side, leading to an increase in the minimum margin. Moreover, Fig.~\ref{fig8}(b) shows that for $\alpha \ge 1.1$, the objective increase vanishes for all three values of $m^q$. This implies that all corresponding stability constraints are slack in this range, so the stability-constrained OPF coincides with the baseline OPF.

\begin{figure}[!t]
\centering

\begin{minipage}{0.48\linewidth}
\centering
\begin{overpic}[width=\linewidth]{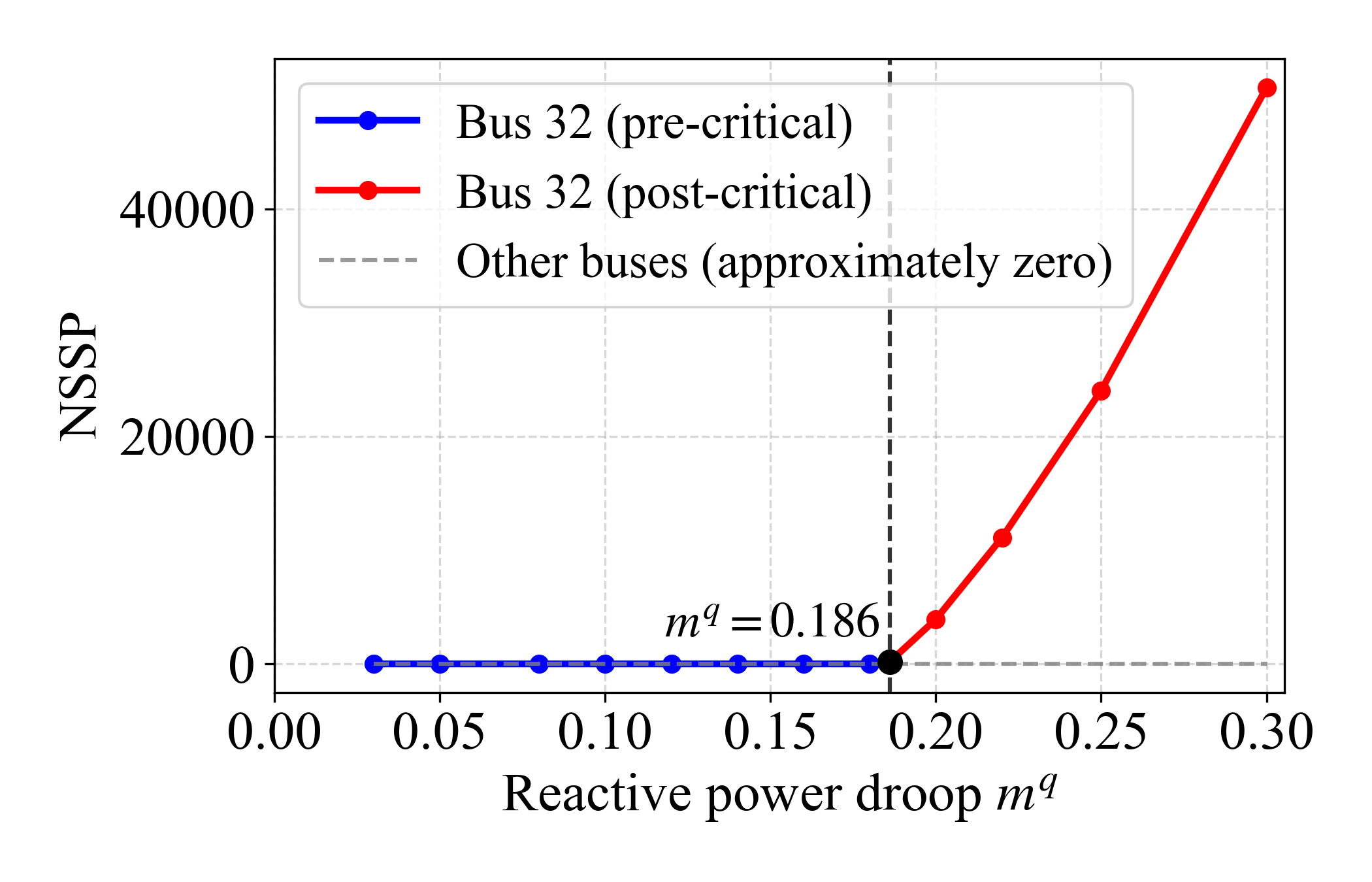}
    \put(0,5.5){\footnotesize (a)}
\end{overpic}
\end{minipage}
\hfill
\begin{minipage}{0.48\linewidth}
\centering
\begin{overpic}[width=\linewidth]{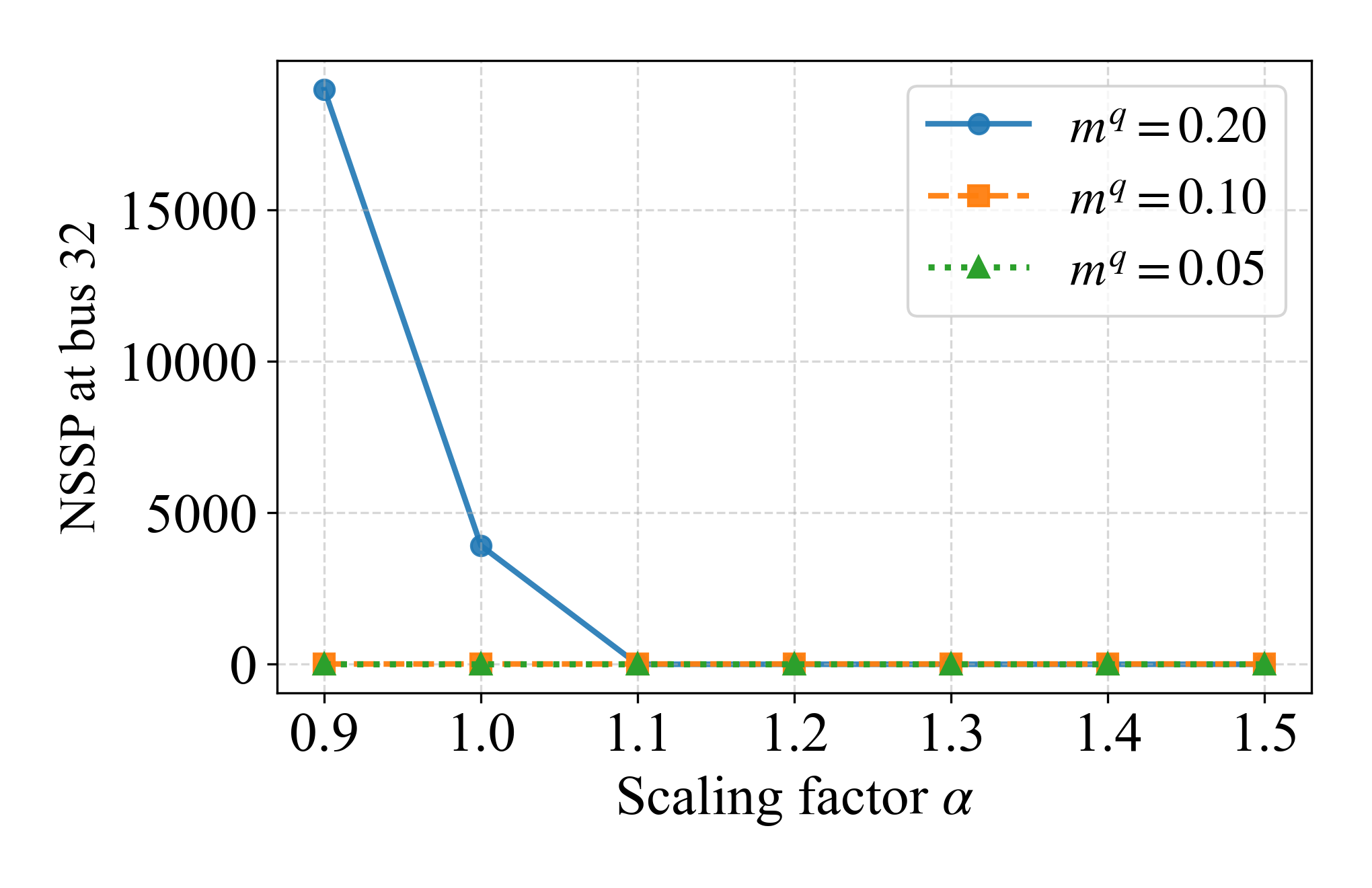}
    \put(0,5.5){\footnotesize (b)}
\end{overpic}
\end{minipage}

\caption{Effects of $m^q$ and $\alpha$ on nodal stability shadow prices: (a) Across all buses versus $m^q$, and (b) At bus~32 versus $\alpha$ for various $m^q$ in the IEEE 39-bus system under a $P$-only objective.}
\label{fig9}
\end{figure}

Figure~\ref{fig9} illustrates the nodal stability shadow prices \eqref{eq:NSSP} under the $P$-only objective. As observed, among the 10 generator buses, only bus 32 exhibits nonzero nodal stability shadow prices, while all other buses remain approximately zero. In Fig.~\ref{fig9}(a), the nodal shadow price at bus 32 remains zero below the critical value $m^q \approx 0.186$, and increases rapidly once $m^q$ exceeds this threshold. This behavior is consistent with Fig.~\ref{fig7}, where the stability constraints become active only beyond the critical value. More specifically, the nonzero nodal shadow price at bus 32 arises from multiple binding stability constraints associated with this bus, such as the pairs $(32,30)$, $(32,36)$, $(32,37)$, and $(32,38)$. These stability constraints are co-binding with several operational constraints, including upper bounds on active and reactive power generation, voltage magnitudes, and branch flow limits. Notably, the shadow prices associated with the stability constraints are observed to be significantly larger than those of the co-binding operational constraints, especially when $m^q$ exceeds the critical value. As $m^q$ increases further, the stability shadow prices grow rapidly and dominate the overall Lagrange multipliers at the critical node. This suggests that bus 32 becomes a critical location where the stability constraints actively interact with the OPF solution, leading to a concentration of nonzero shadow prices at this node. Fig.~\ref{fig9}(b) further shows the evolution of the nodal shadow price at bus 32 with respect to the network scaling factor $\alpha$. For $m^q=0.2$, the nodal shadow price is positive when $\alpha$ is small, and decreases to zero as $\alpha$ increases. This aligns with Fig.~\ref{fig8}, where the stability constraints become slack for sufficiently large $\alpha$. For $m^q=0.1$ and $m^q=0.05$, the nodal shadow prices remain zero across all values of $\alpha$, consistent with the fact that the stability constraints are inactive. These results demonstrate that the nodal stability shadow price provides a localized measure of how stability constraints impact the OPF solution, and that such impacts can be highly concentrated on a small subset of critical buses. In particular, even when multiple constraints are co-binding, the stability constraints can dominate the marginal cost, highlighting their critical role in shaping the optimal operating point.

\begin{figure}[!t]
\centering

\begin{minipage}{0.48\linewidth}
\centering
\begin{overpic}[width=\linewidth]{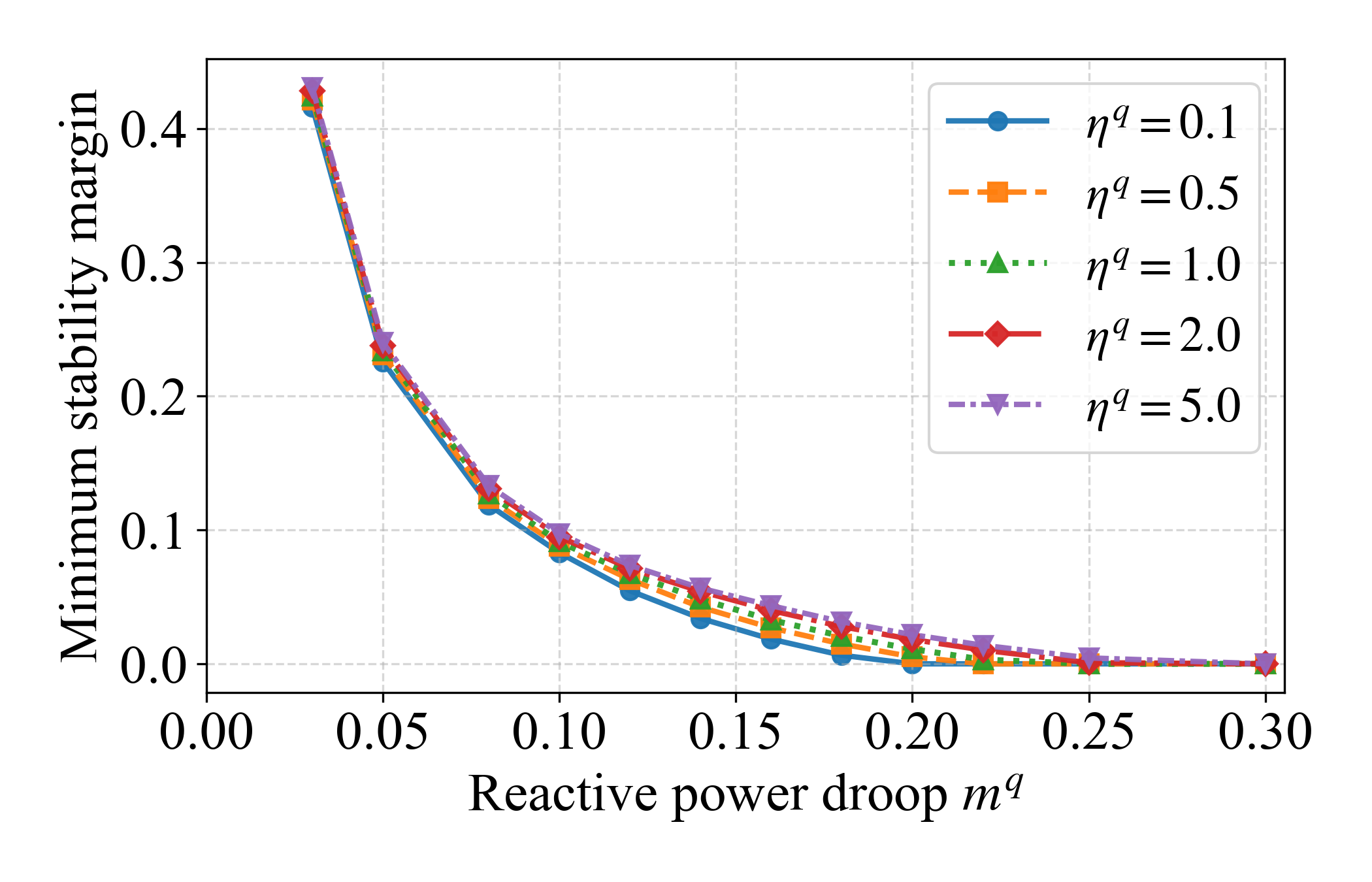}
    \put(0,5.5){\footnotesize (a)}
\end{overpic}
\end{minipage}
\hfill
\begin{minipage}{0.48\linewidth}
\centering
\begin{overpic}[width=\linewidth]{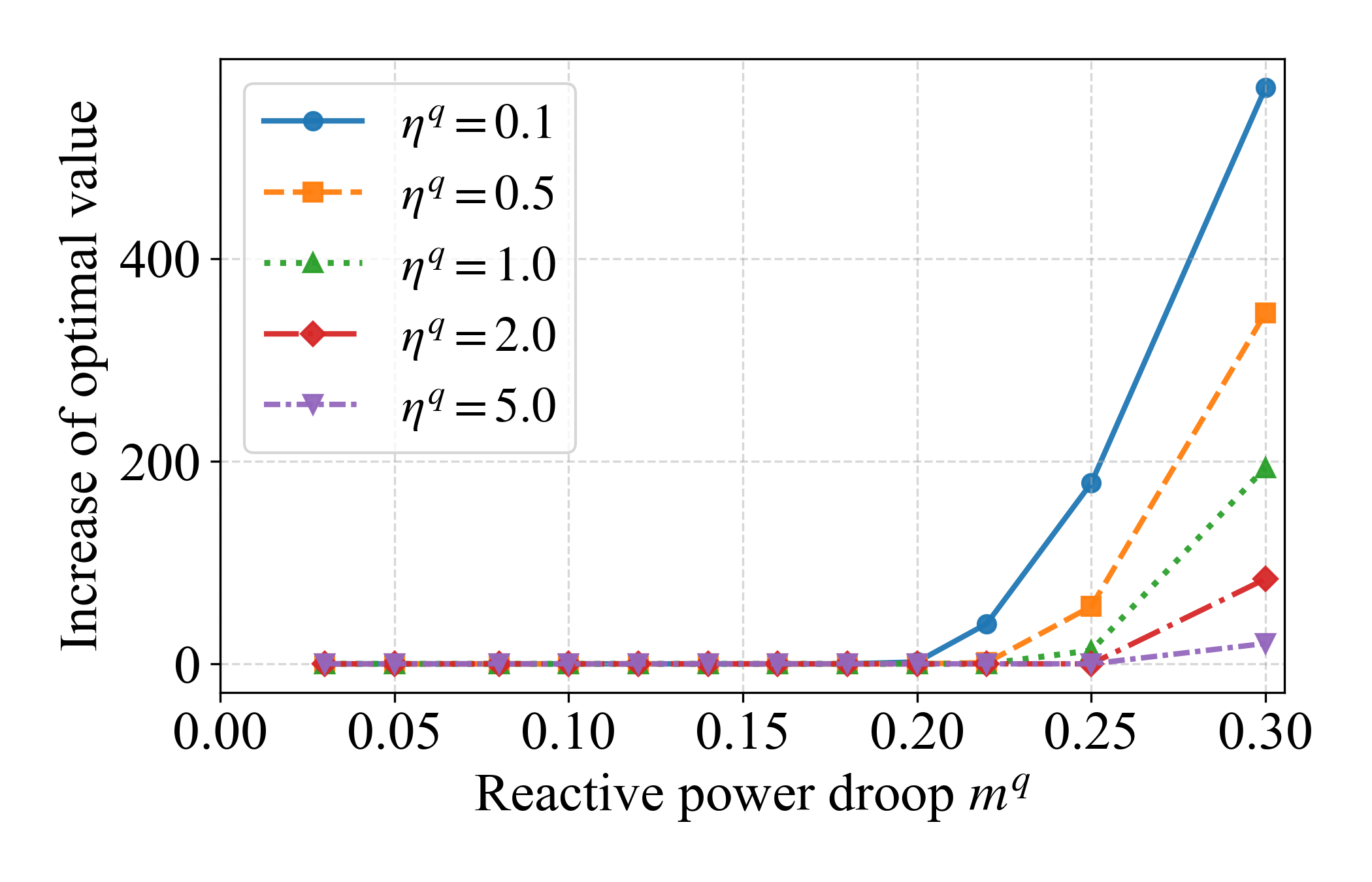}
    \put(0,5.5){\footnotesize (b)}
\end{overpic}
\end{minipage}

\caption{Effect of $m^q$ and $\eta^q$ on: (a) Minimum stability margin, and (b) Optimal value increase in the IEEE 39-bus system with reactive power costs.}
\label{fig10}

\end{figure}

We next consider an objective function \eqref{eq:general_cost} that incorporates a cost on reactive power with $d_i>0$. To parameterize the relative weight of reactive power cost, we set
$d_i = \eta^q c_i, \ \forall i \in \mathcal G$,
where $\eta^q$ represents the cost ratio between the quadratic terms of reactive and active power. In the following, $\eta^q$ is treated as a tuning parameter to examine how the inclusion of reactive power in the objective influences the stability constraints and their associated shadow prices.

Figure~\ref{fig10} shows the effect of incorporating reactive power into the objective function for different values of $\eta^q$. In Fig.~\ref{fig10}(a), the minimum stability margin decreases with $m^q$ for all values of $\eta^q$, consistent with the $P$-only case. However, compared to Fig.~\ref{fig7}, the margin curves are consistently shifted upward as $\eta^q$ increases. This upward shift suggests that the critical value of $m^q$ at which the stability constraints become active is increased. This effect is more clearly illustrated in Fig.~\ref{fig10}(b), where the objective remains unchanged up to a threshold value of $m^q$, and beyond which it increases sharply. As $\eta^q$ increases, this threshold shifts to larger values of $m^q$, providing direct evidence that the activation point of the stability constraints is postponed. This behavior can be interpreted as follows. The inclusion of the term $d_i Q_{G,i}^2$ penalizes large reactive power injections, which alters the optimal operating point of the OPF problem. In particular, it discourages solutions that rely on highly uneven reactive power dispatch across generators. As a result, the resulting voltage profile across generator buses tends to be more balanced, leading to reduced voltage differences between buses as shown in 
Fig.~\ref{fig:voltage_profile} and 
Fig.~\ref{fig:voltage_profile2}. This, in turn, increases the stability margin and delays the activation of the stability constraints. Moreover, for larger values of $\eta^q$, the objective increase beyond the threshold is significantly reduced. This indicates that incorporating reactive power cost not only postpones the activation of the stability constraints, but also mitigates their economic impact once they become active. These results show that including reactive power in the objective function effectively enlarges the admissible range of $m^q$ for which the stability constraints remain inactive.


\begin{figure}[t]
    \centering
    \includegraphics[width=0.8\linewidth]{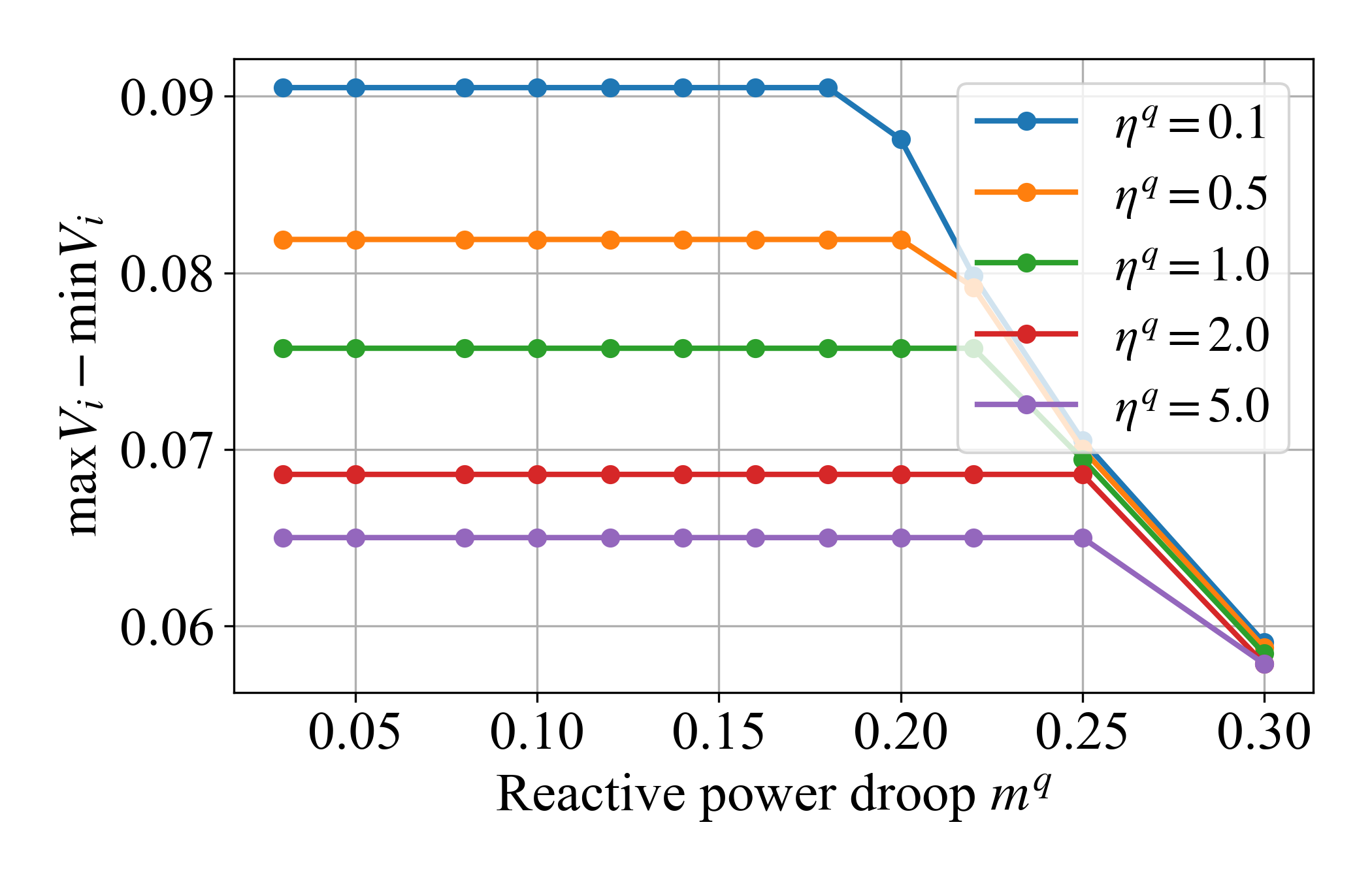}
    \caption{Maximum voltage difference across generator buses ($\max_i V_i - \min_i V_i$) under different values of $m^q$ and $\eta^q$.}
    \label{fig:voltage_profile}
\end{figure}

\begin{figure}[t]
    \centering
    \includegraphics[width=0.8\linewidth]{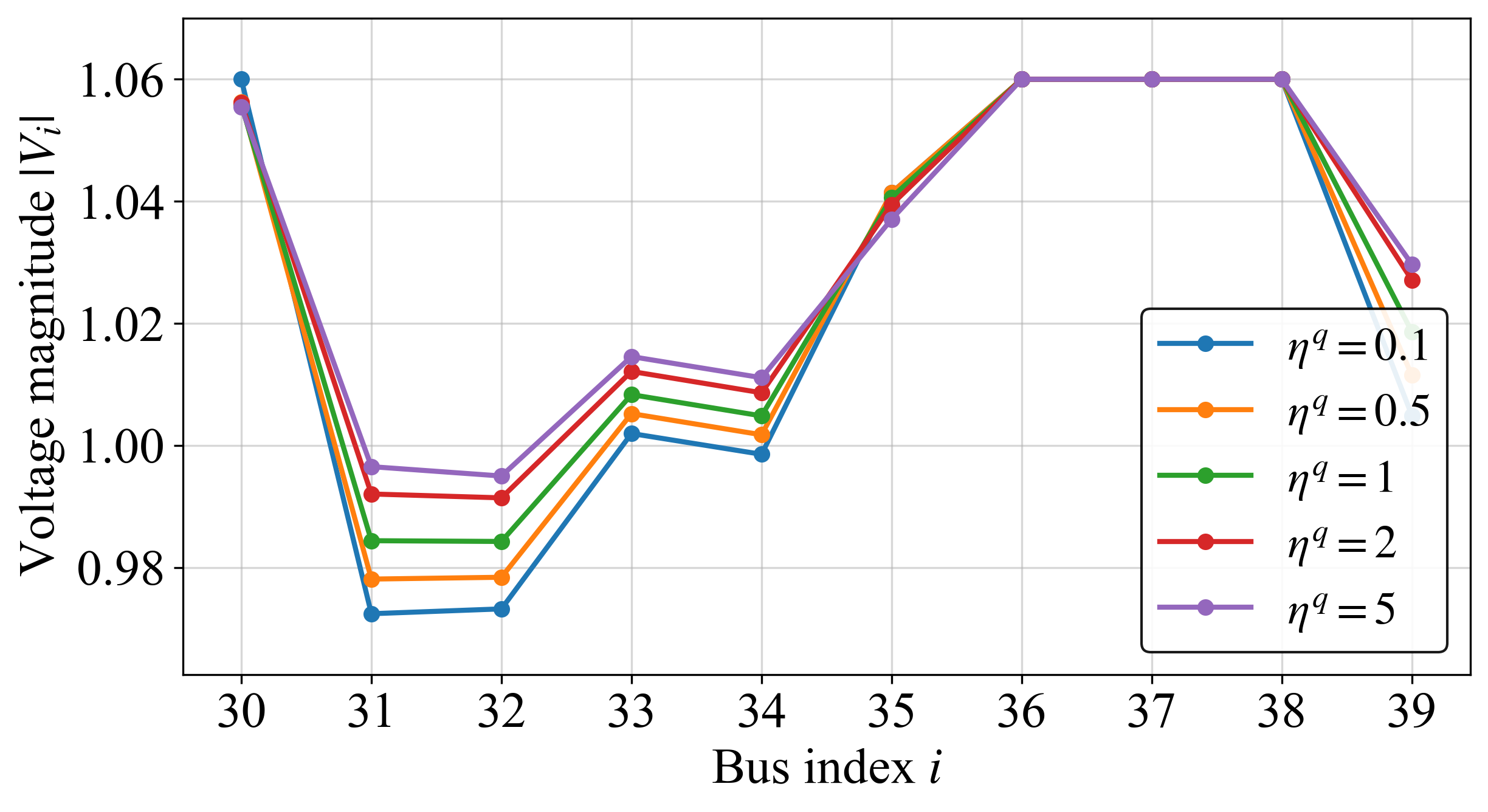}
    \caption{Voltage profiles across generator buses (30--39) for different values of $\eta^q$, with $m^q$ fixed at 0.2.}
    \label{fig:voltage_profile2}
\end{figure}

\begin{figure}[!t]
\centering

\begin{minipage}{0.48\linewidth}
\centering
\begin{overpic}[width=\linewidth]{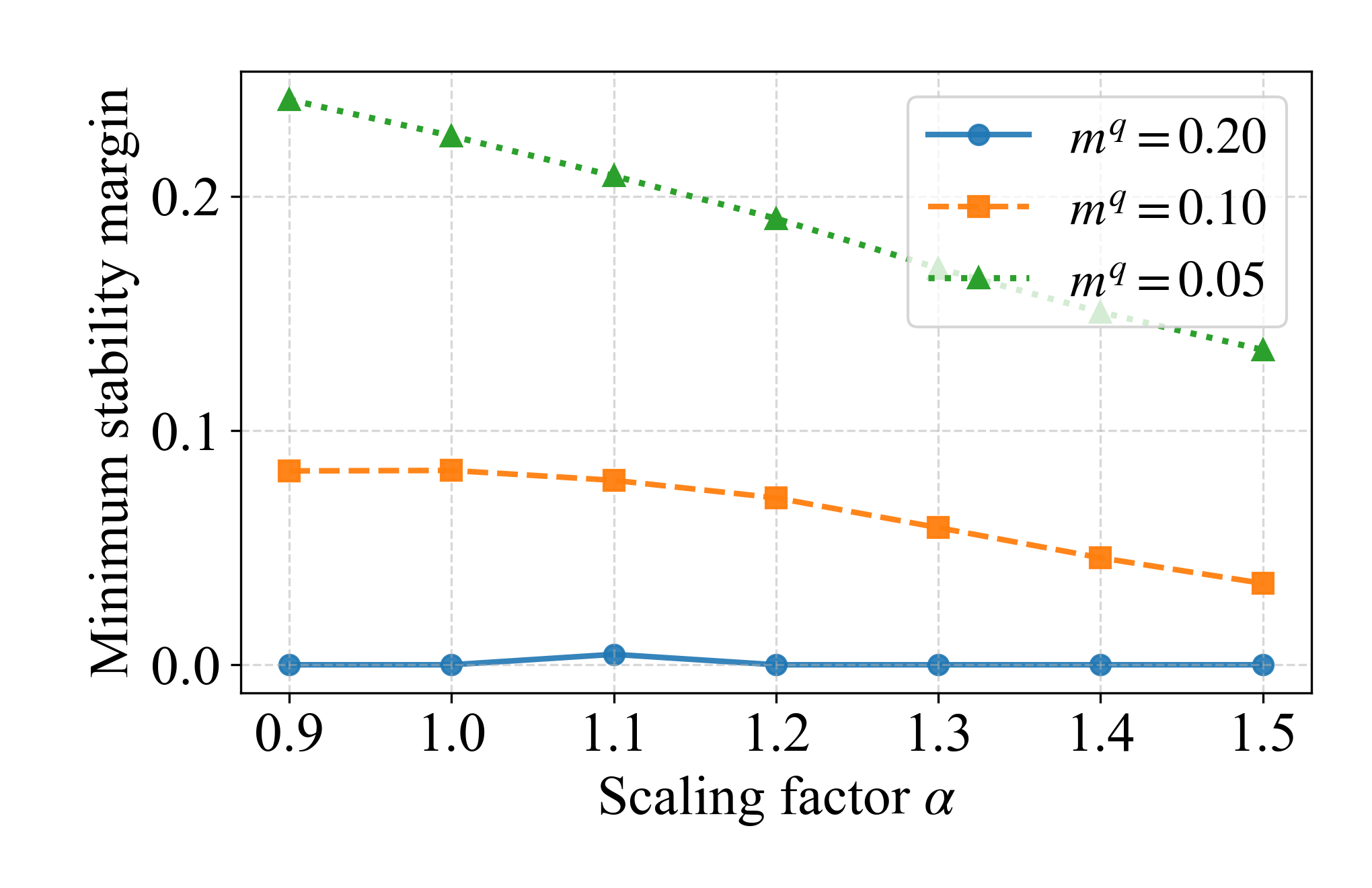}
    \put(0,5.5){\footnotesize (a)}
\end{overpic}
\end{minipage}
\hfill
\begin{minipage}{0.48\linewidth}
\centering
\begin{overpic}[width=\linewidth]{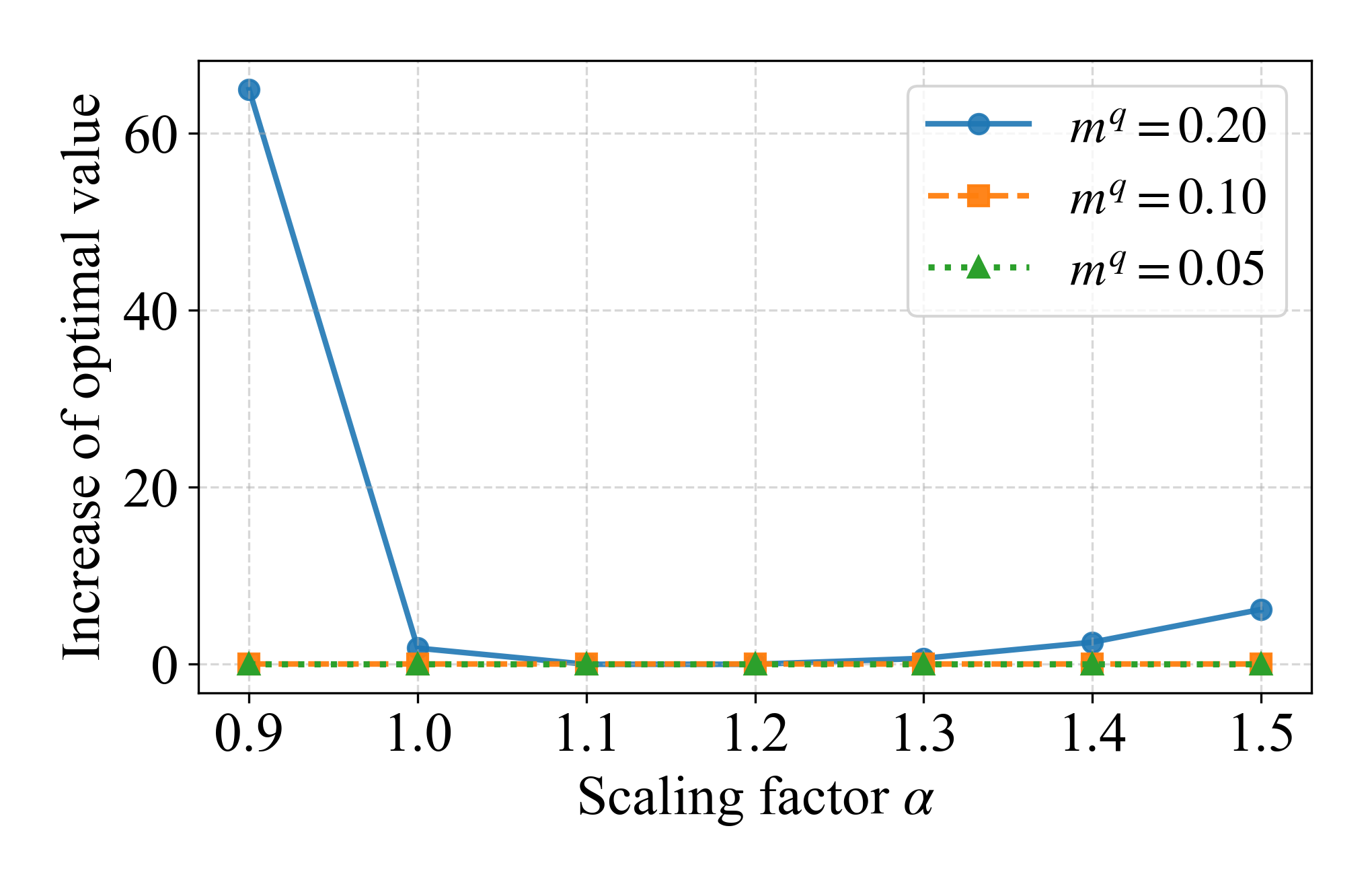}
    \put(0,5.5){\footnotesize (b)}
\end{overpic}
\end{minipage}

\caption{Effects of $m^q$ and $\alpha$ on:
(a) Minimum stability margin, and (b) Increase of optimal value in the IEEE 39-bus system with reactive power costs.}
\label{fig12}

\end{figure}

Figure~\ref{fig12} shows the effect of the network scaling factor $\alpha$ for different values of $m^q$ under the objective function that includes reactive power cost. For $m^q=0.05$ and $m^q=0.1$, the minimum stability margin remains positive across all values of $\alpha$ and decreases gradually as $\alpha$ increases. Correspondingly, the objective increase remains zero, indicating that the stability constraints are inactive throughout. In contrast, for $m^q=0.2$, a non-monotonic behavior is observed. The minimum margin is zero for small values of $\alpha$, becomes positive around $\alpha \approx 1.1$, and then returns to zero as $\alpha$ increases further. This is reflected more clearly in the U-shaped profile of the objective increase, which first decreases to zero and then increases again. Therefore, the stability constraints are active for small $\alpha$, become temporarily inactive near $\alpha \approx 1.1$, and become active again for larger $\alpha$. This non-monotonic behavior is primarily due to the fact that the susceptance matrix $B$ enters both the stability constraints and the power flow equations. As a result, scaling $\alpha$ simultaneously affects the constraint boundary and the optimal operating point. The minimum margin is therefore determined by the interplay between the tightening of the constraint and the changes in the voltage profile induced by the re-optimized solution. In particular, the temporary deactivation occurs when the change in the operating point reduces the voltage differences sufficiently to offset the tightening of the constraint. The inclusion of reactive power cost further modifies this behavior by altering the structure of the optimal solution. In particular, it changes how the operating point responds to variations in $\alpha$, making the interaction between constraint tightening and solution reconfiguration more pronounced. This mechanism is fundamentally different from the effect of $m^q$. While $m^q$ only rescales the stability constraints and leads to a monotonic transition from inactive to active constraints, $\alpha$ modifies both the constraint boundary and the optimal solution, leading to qualitatively different system responses.

\begin{figure}[!t]
\centering

\begin{minipage}{0.48\linewidth}
\centering
\begin{overpic}[width=\linewidth]{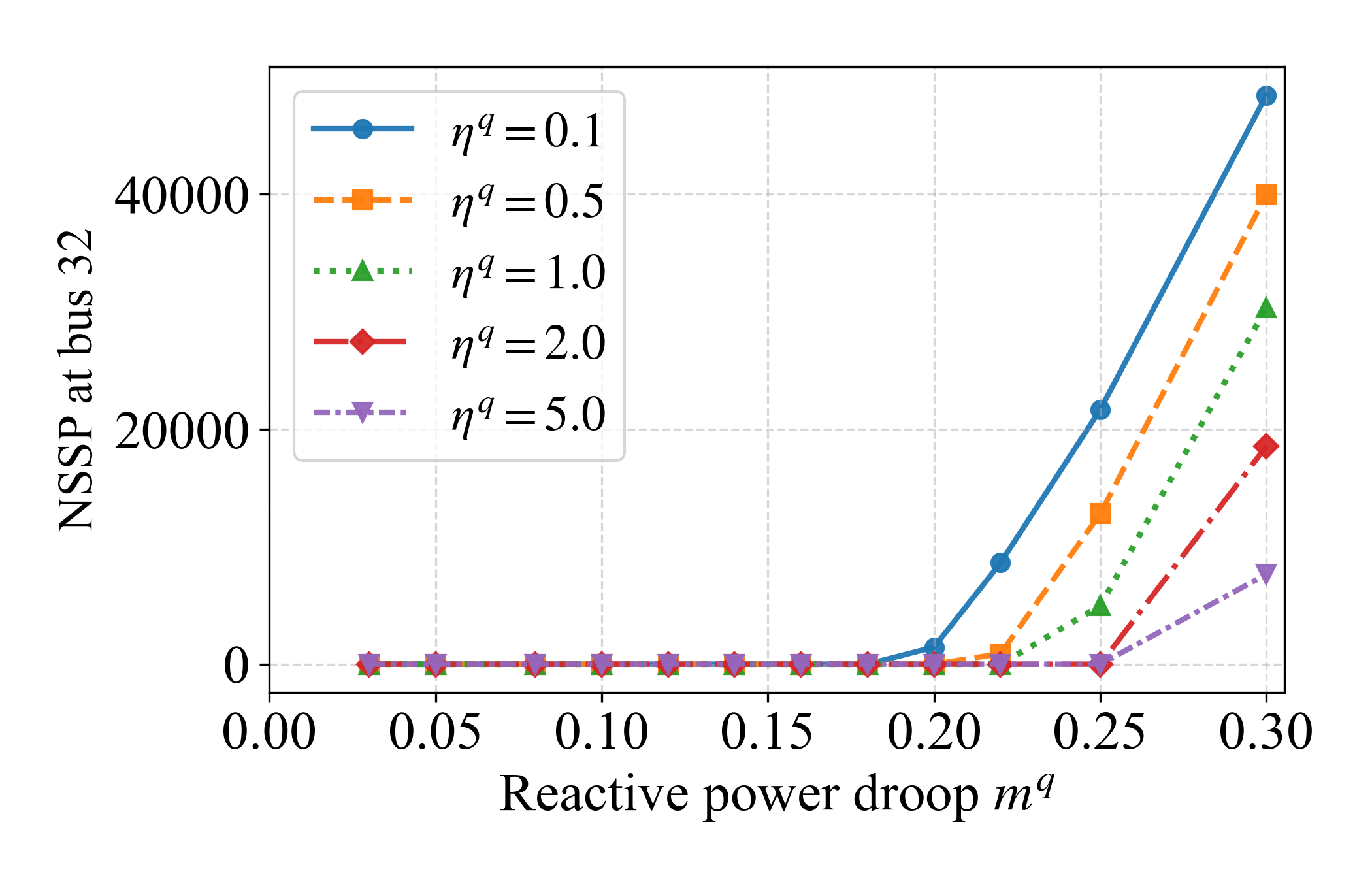}
    \put(0,5.5){\footnotesize (a)}
\end{overpic}
\end{minipage}
\hfill
\begin{minipage}{0.48\linewidth}
\centering
\begin{overpic}[width=\linewidth]{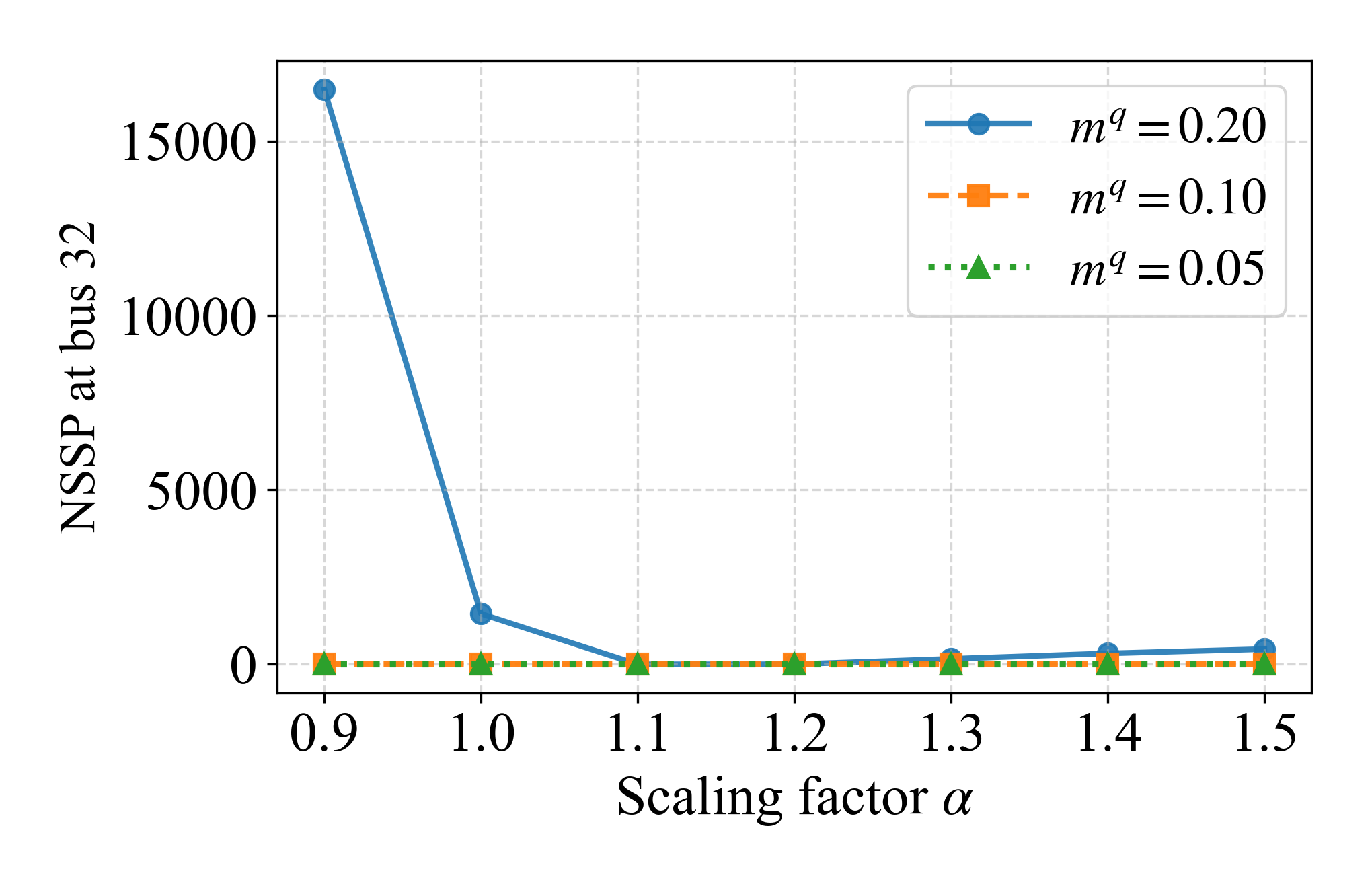}
    \put(0,5.5){\footnotesize (b)}
\end{overpic}
\end{minipage}

\caption{Effects of $m^q$, $\alpha$ and $\eta^q$ on the nodal stability shadow price at bus~32:
(a) Versus $m^q,\eta^q$, and (b) Versus $m^q,\alpha$ in the IEEE 39-bus system with reactive power costs.}
\label{fig14}

\end{figure}

Figure~\ref{fig14} illustrates the nodal stability shadow price \eqref{eq:NSSP} under the objective function that includes reactive power cost. 
As shown in Fig.~\ref{fig14}(a), the nodal shadow price at bus 32 remains essentially zero for small values of $m^q$, indicating that the stability constraints are inactive in this range. Once $m^q$ exceeds a critical value, the nodal shadow price increases rapidly, reflecting the activation of multiple stability constraints. This sharp increase indicates that further tightening of the stability constraints would incur a significant increase in the objective value. Moreover, larger values of $\eta^q$ lead to smaller shadow prices, suggesting that incorporating reactive power cost alleviates the marginal burden of the stability constraints. Fig.~\ref{fig14}(b) shows the variation of the nodal shadow price with respect to $\alpha$. For $m^q=0.2$, a large shadow price is observed at $\alpha=0.9$, where the stability constraints are strongly binding. As $\alpha$ increases, the shadow price drops rapidly to zero around $\alpha \approx 1.1$, indicating that the constraints become inactive. For larger values of $\alpha$, the shadow price increases again, reflecting the reactivation of the stability constraints. This non-monotonic (U-shaped) behavior is consistent with the observations in Fig.~\ref{fig12}, where the stability constraints exhibit a transition from active to inactive and back to active. These results are consistent with the earlier observations on stability margin and objective increase. While Fig.~\ref{fig12} characterizes whether the stability constraints are active,
Fig.~\ref{fig14} quantifies the marginal cost associated with enforcing the stability constraints. In particular, the nodal shadow price peaks in the most restrictive ranges and vanishes when the constraints become inactive, confirming that the non-monotonic behavior observed in the objective is directly reflected in the marginal prices of the stability constraints.

\section{Conclusion}

This paper proposed a decentralized (small-signal) stability-constrained OPF framework for inverter-based power systems. We identified a class of algebraic decentralized stability conditions that depend on local voltage differences, which are well suited for incorporation into steady-state optimization. 
From a theoretical perspective, we characterized the role of the decentralized stability constraints through their associated shadow prices. The concept of Nodal Stability Shadow Price (NSSP) is proposed that reflects the individual economic contribution of stability from each inverter. We proved that in lossless networks with active-power-only cost, stability constraints may be binding but the shadow prices are zero,
revealing a degeneracy case between stability and the objective structure. In practice, inverter-based resources (IBRs) have limited capacity and reactive power has an opportunity cost. By incorporating cost on reactive power in the objective, stability constraints admit positive shadow prices, and practical economic interpretations of stability constraints are provided. Numerical experiments on a two-bus system and the IEEE 39-bus system validate the proposed formulation and theoretical findings, illustrating the impact of stability constraints on the feasible region and optimal solutions, together with the important role of reactive power droop parameter. Future work will focus on extending the proposed framework to more general network models with losses, as well as exploring alternative stability criteria. In addition, it is of interest to investigate the co-optimization of control parameters and dynamic voltage support from IBRs in operational settings while considering various practical constraints. 

\appendix
\section*{Equivalence Between Max- and Split-Form Stability Constraints}

For each inverter bus $i\in\mathcal G$, define
\[
\psi_{ij}(V):=V_j-V_i-\Gamma_i,\qquad
j\in\mathcal N_i^{\mathrm{red}},
\]
and
\[
\psi_i(V):=\max_{j\in\mathcal N_i^{\mathrm{red}}}\psi_{ij}(V).
\]
Then the decentralized stability condition can be imposed either in the max form
\[
\psi_i(V)\le 0,\qquad \forall i\in\mathcal G,
\]
or in the split form
\[
\psi_{ij}(V)\le 0,\qquad \forall i\in\mathcal G,\ \forall j\in\mathcal N_i^{\mathrm{red}}.
\]

Consider the following two OPF formulations, which differ only in the representation of the stability constraints.

\medskip
\noindent\textbf{Max-form OPF:}
\begin{equation}
\begin{aligned}
\min_{x}\quad & J(x) \\
\text{s.t.}\quad
& g(x)=0,\\
& q(x)\le 0,\\
& \psi_i(V)\le 0,\qquad \forall i\in\mathcal G,
\end{aligned}
\label{eq:opf_max_form}
\end{equation}
where $x$ collects all decision variables, $g(x)=0$ denotes the equality constraints, and $q(x)\le 0$ denotes all inequality constraints other than the stability constraints.

\medskip
\noindent\textbf{Split-form OPF:}
\begin{equation}
\begin{aligned}
\min_{x}\quad & J(x) \\
\text{s.t.}\quad
& g(x)=0,\\
& q(x)\le 0,\\
& \psi_{ij}(V)\le 0,\qquad \forall i\in\mathcal G,\ \forall j\in\mathcal N_i^{\mathrm{red}}.
\end{aligned}
\label{eq:opf_split_form}
\end{equation}

\begin{proposition}
\label{prop:max_split_equivalence}
The max-form OPF \eqref{eq:opf_max_form} and the split-form OPF \eqref{eq:opf_split_form} have the same feasible set. Moreover, a feasible point is a KKT point of one formulation if and only if it is a KKT point of the other formulation. At any common KKT point, the associated stability multipliers can be chosen to satisfy
\[
\lambda_i^{\mathrm{stab}}
=
\sum_{j\in\mathcal N_i^{\mathrm{red}}}\lambda_{ij}^{\mathrm{stab}},
\qquad \forall i\in\mathcal G.
\]
\end{proposition}

\noindent\textit{Proof:}
Fix any $i\in\mathcal G$. Suppose first that $\psi_i(V)\le 0$. Since
\[
\psi_i(V)=\max_{j\in\mathcal N_i^{\mathrm{red}}}\psi_{ij}(V),
\]
it follows that, for every $j\in\mathcal N_i^{\mathrm{red}}$,
\[
\psi_{ij}(V)\le \max_{k\in\mathcal N_i^{\mathrm{red}}}\psi_{ik}(V)=\psi_i(V)\le 0.
\]
Conversely, suppose that
\[
\psi_{ij}(V)\le 0,\qquad \forall j\in\mathcal N_i^{\mathrm{red}}.
\]
Then
\[
\psi_i(V)=\max_{j\in\mathcal N_i^{\mathrm{red}}}\psi_{ij}(V)\le 0.
\]
Therefore, for each $i\in\mathcal G$,
\[
\psi_i(V)\le 0
\quad\Longleftrightarrow\quad
\psi_{ij}(V)\le 0,\qquad \forall j\in\mathcal N_i^{\mathrm{red}}.
\]
Since this holds for every $i\in\mathcal G$, the two formulations have the same feasible set.

\textcolor{black}{We next show that a feasible point is a KKT point of the max-form OPF if and only if it is a KKT point of the split-form OPF.}

\noindent\textbf{(i) From max form to split form.} Assume first that $x^\star$ is a KKT point of the max-form OPF \eqref{eq:opf_max_form}. Then there exist multipliers $\lambda^\star$ for $g(x)=0$, $\nu^\star\ge 0$ for $q(x)\le 0$, and $\lambda_i^{\mathrm{stab}}\ge 0$ for $\psi_i(V)\le 0$, such that primal feasibility, dual feasibility, and complementary slackness hold, together with the stationarity condition
\begin{equation}
0\in
\nabla J(x^\star)
+\nabla g(x^\star)^\top\lambda^\star
+\nabla q(x^\star)^\top\nu^\star
+\sum_{i\in\mathcal G}\lambda_i^{\mathrm{stab}}\,\partial\psi_i(V^\star).
\label{eq:max_stationarity}
\end{equation}
Here, $\partial\psi_i$ is the subdifferential of $\psi_i$ and is given by
\[
\partial\psi_i(V^\star)
=
\operatorname{conv}
\left\{
\nabla\psi_{ij}(V^\star):
j\in\mathcal A_i(V^\star)
\right\},
\]
where $\operatorname{conv}\{\cdot\}$ denotes the convex hull of a set, and $\mathcal A_i(V^\star)$ is the set of indices attaining the maximum, i.e., 
\[
\mathcal A_i(V^\star):=
\left\{
j\in\mathcal N_i^{\mathrm{red}}:
\psi_{ij}(V^\star)=\psi_i(V^\star)
\right\}.
\]


Hence, for each $i\in\mathcal G$, there exists convex combination coefficients $\alpha_{ij}\ge 0$, such that  
\[
\zeta_i^\star=
\sum_{j\in\mathcal A_i(V^\star)}
\alpha_{ij}\nabla\psi_{ij}(V^\star)\in\partial\psi_i(V^\star),
\]
and \eqref{eq:max_stationarity} becomes
\begin{equation}
0=
\nabla J(x^\star)
+\nabla g(x^\star)^\top\lambda^\star
+\nabla q(x^\star)^\top\nu^\star
+\sum_{i\in\mathcal G}\lambda_i^{\mathrm{stab}}\zeta_i^\star.
\label{eq:max_stationarity_selected}
\end{equation}

Now define
\[
\lambda_{ij}^{\mathrm{stab}}:=
\begin{cases}
\lambda_i^{\mathrm{stab}}\alpha_{ij},
& j\in\mathcal A_i(V^\star),\\[1mm]
0,
& j\notin\mathcal A_i(V^\star).
\end{cases}
\]
Then $\lambda_{ij}^{\mathrm{stab}}\ge 0$ for all $i,j$, and
\[
\sum_{j\in\mathcal N_i^{\mathrm{red}}}
\lambda_{ij}^{\mathrm{stab}}\nabla\psi_{ij}(V^\star)
=
\lambda_i^{\mathrm{stab}}\zeta_i^\star.
\]
Substituting this identity into \eqref{eq:max_stationarity_selected} gives
\begin{align}
0=
\nabla J(x^\star)
&+\nabla g(x^\star)^\top\lambda^\star
+\nabla q(x^\star)^\top\nu^\star\nonumber\\
&+\sum_{i\in\mathcal G}\sum_{j\in\mathcal N_i^{\mathrm{red}}}
\lambda_{ij}^{\mathrm{stab}}\nabla\psi_{ij}(V^\star),\nonumber
\end{align}
which is exactly the stationarity condition for the split-form OPF \eqref{eq:opf_split_form}.

It remains to verify complementary slackness for the split-form stability multipliers. If $j\notin\mathcal A_i(V^\star)$, then $\lambda_{ij}^{\mathrm{stab}}=0$ by construction, so
\[
\lambda_{ij}^{\mathrm{stab}}\psi_{ij}(V^\star)=0.
\]
If $j\in\mathcal A_i(V^\star)$ and $\lambda_i^{\mathrm{stab}}=0$, then again $\lambda_{ij}^{\mathrm{stab}}=0$, so the same conclusion holds. Finally, if $j\in\mathcal A_i(V^\star)$ and $\lambda_i^{\mathrm{stab}}>0$, then complementary slackness for the max-form constraint implies
\[
\lambda_i^{\mathrm{stab}}\psi_i(V^\star)=0,
\]
hence
\[
\psi_i(V^\star)=0.
\]
Since $j\in\mathcal A_i(V^\star)$, we also have
\[
\psi_{ij}(V^\star)=\psi_i(V^\star)=0.
\]
Therefore,
\[
\lambda_{ij}^{\mathrm{stab}}\psi_{ij}(V^\star)=0.
\]
Thus all KKT conditions of the split-form OPF are satisfied, and $x^\star$ is a KKT point of \eqref{eq:opf_split_form}.

\noindent\textbf{(ii) From split form to max form.}
We now prove the converse direction. Assume that $x^\star$ is a KKT point of the split-form OPF \eqref{eq:opf_split_form}. Then there exist multipliers $\lambda^\star$ for $g(x)=0$, $\nu^\star\ge 0$ for $q(x)\le 0$, and $\lambda_{ij}^{\mathrm{stab}}\ge 0$ for $\psi_{ij}(V)\le 0$, such that primal feasibility, dual feasibility, and complementary slackness hold, together with
\begin{align}
0=
\nabla J(x^\star)
&+\nabla g(x^\star)^\top\lambda^\star
+\nabla q(x^\star)^\top\nu^\star\nonumber\\
&+\sum_{i\in\mathcal G}\sum_{j\in\mathcal N_i^{\mathrm{red}}}
\lambda_{ij}^{\mathrm{stab}}\nabla\psi_{ij}(V^\star).
\label{eq:split_stationarity}
\end{align}

For each $i\in\mathcal G$, define
\[
\lambda_i^{\mathrm{stab}}
:=
\sum_{j\in\mathcal N_i^{\mathrm{red}}}\lambda_{ij}^{\mathrm{stab}}.
\]
If $\lambda_i^{\mathrm{stab}}=0$, then all $\lambda_{ij}^{\mathrm{stab}}=0$, and the contribution of bus $i$ to \eqref{eq:split_stationarity} is zero. Suppose now that $\lambda_i^{\mathrm{stab}}>0$, and define
\[
\alpha_{ij}:=
\frac{\lambda_{ij}^{\mathrm{stab}}}{\lambda_i^{\mathrm{stab}}},
\qquad
j\in\mathcal N_i^{\mathrm{red}}.
\]
Then $\alpha_{ij}\ge 0$ and
\[
\sum_{j\in\mathcal N_i^{\mathrm{red}}}\alpha_{ij}=1.
\]

Next, by complementary slackness for the split-form problem,
\[
\lambda_{ij}^{\mathrm{stab}}\psi_{ij}(V^\star)=0,
\qquad
\forall j\in\mathcal N_i^{\mathrm{red}}.
\]
Therefore, if $\lambda_{ij}^{\mathrm{stab}}>0$, then
\[
\psi_{ij}(V^\star)=0.
\]
Since $x^\star$ is feasible for the split-form problem, all $\psi_{ij}(V^\star)\le 0$, and hence
\[
\psi_i(V^\star)=\max_{j\in\mathcal N_i^{\mathrm{red}}}\psi_{ij}(V^\star)\le 0.
\]
Moreover, any index $j$ with $\lambda_{ij}^{\mathrm{stab}}>0$ must attain this maximum, because it satisfies $\psi_{ij}(V^\star)=0$. Therefore,
\[
\sum_{j\in\mathcal N_i^{\mathrm{red}}}\!
\alpha_{ij}\nabla\psi_{ij}(V^\star)
\!\in\!
\operatorname{conv}
\left\{
\nabla\psi_{ij}(V^\star)\!:\!
j\!\in\!\mathcal A_i(V^\star)
\right\}
\!=\!
\partial\psi_i(V^\star).
\]
Thus there exists
\[
\zeta_i^\star\in\partial\psi_i(V^\star)
\]
such that
\[
\sum_{j\in\mathcal N_i^{\mathrm{red}}}
\lambda_{ij}^{\mathrm{stab}}\nabla\psi_{ij}(V^\star)
=
\lambda_i^{\mathrm{stab}}\zeta_i^\star.
\]
If $\lambda_i^{\mathrm{stab}}=0$, the same identity holds trivially with zero contribution. Substituting these identities into \eqref{eq:split_stationarity}, we obtain
\[
0\in
\nabla J(x^\star)
+\nabla g(x^\star)^\top\lambda^\star
+\nabla q(x^\star)^\top\nu^\star
+\sum_{i\in\mathcal G}\lambda_i^{\mathrm{stab}}\partial\psi_i(V^\star),
\]
which is exactly the stationarity condition for the max-form OPF \eqref{eq:opf_max_form}.

It remains to verify complementary slackness for the max-form stability multipliers. If $\lambda_i^{\mathrm{stab}}=0$, then
\[
\lambda_i^{\mathrm{stab}}\psi_i(V^\star)=0
\]
holds trivially. Suppose now that $\lambda_i^{\mathrm{stab}}>0$. Then at least one $j\in\mathcal N_i^{\mathrm{red}}$ satisfies $\lambda_{ij}^{\mathrm{stab}}>0$, which implies
\[
\psi_{ij}(V^\star)=0.
\]
Since all $\psi_{ij}(V^\star)\le 0$, it follows that
\[
\psi_i(V^\star)=\max_{j\in\mathcal N_i^{\mathrm{red}}}\psi_{ij}(V^\star)=0.
\]
Therefore,
\[
\lambda_i^{\mathrm{stab}}\psi_i(V^\star)=0.
\]
Thus all KKT conditions of the max-form OPF are satisfied, and $x^\star$ is a KKT point of \eqref{eq:opf_max_form}.

\textcolor{black}{Finally, we characterize the relationship between the stability multipliers.}
From the construction in the max-form to split-form direction, we have
\[
\lambda_{ij}^{\mathrm{stab}} = \lambda_i^{\mathrm{stab}}\alpha_{ij},
\qquad
\sum_{j\in\mathcal A_i(V^\star)} \alpha_{ij}=1,
\]
which directly implies
\[
\sum_{j\in\mathcal N_i^{\mathrm{red}}} \lambda_{ij}^{\mathrm{stab}}
=
\lambda_i^{\mathrm{stab}}.
\]

Conversely, in the split-form to max-form direction, the multiplier
$\lambda_i^{\mathrm{stab}}$ is defined as
\[
\lambda_i^{\mathrm{stab}}
=
\sum_{j\in\mathcal N_i^{\mathrm{red}}}\lambda_{ij}^{\mathrm{stab}}.
\]

Therefore, at any common KKT point, the stability multipliers of the two
formulations can be chosen to satisfy
\[
\lambda_i^{\mathrm{stab}}
=
\sum_{j\in\mathcal N_i^{\mathrm{red}}}\lambda_{ij}^{\mathrm{stab}},
\qquad \forall i\in\mathcal G. 
\hspace{3.5cm} \blacksquare
\]

\bibliographystyle{IEEEtran}
\bibliography{reference}

\end{document}